\documentclass[11pt]{article}

\usepackage{latexsym}
\usepackage{amssymb}
\usepackage{amsthm}
\usepackage{amscd}
\usepackage{amsmath}
\usepackage{tikz}

\usepgflibrary{arrows}

\newtheorem{theorem}{Theorem}[section]

\newtheorem{proposition}[theorem]{Proposition}
\newtheorem{corollary}[theorem]{Corollary}

\theoremstyle{definition}
\newtheorem*{example}{Example}
\newtheorem*{remark}{Remark}
\newtheorem*{remarks}{Remarks}

\numberwithin{equation}{section}

\newcommand{\FF}{\mathbb{F}}  \def\QQ{\mathbb{Q}}  \newcommand{\CC}{\mathbb{C}} 
   
\def\tr{\mathrm{tr}} \def\dim{\mathrm{dim}} \def\Hom{\mathrm{Hom}} 
 \def\ZZ{\mathbb{Z}} 
    \def\GL{\mathrm{GL}}   \def\Res{\mathrm{Res}} \def\dd{\displaystyle} \def\Inf{\mathrm{Inf}} \def\ch{\mathrm{ch}} \def\spanning{\textnormal{-span}}
\def\vphi{\varphi}  
 \def\veps{\varepsilon} 
  
  \def\scscs{\scriptscriptstyle} \def\scs{\scriptstyle}

\newcommand{\cX}{\mathcal{X}}
\newcommand{\cS}{\mathcal{S}}

\newcommand{\cK}{\mathcal{K}}
\newcommand{\cM}{\mathcal{M}}

\newcommand{\cL}{\mathcal{L}}

\newcommand{\cA}{\mathcal{A}}
\newcommand{\One}{{1\hspace{-.14cm} 1}}

\newcommand{\Id}{\mathrm{Id}}

\newcommand{\NCSym}{\mathrm{\mathbf{\Pi}}}
\newcommand{\WSym}{\mathrm{\mathbf{\Pi}}}

\newcommand{\UT}{\mathrm{UT}}
\newcommand{\SC}{\mathrm{\mathbf{SC}}}
\newcommand{\SInd}{\mathrm{SInd}}
\newcommand{\Def}{\mathrm{Def}}
\newcommand{\fkn}{\mathfrak{n}}
\newcommand{\st}{\mathrm{st}}
\newcommand{\KK}{\mathbb{K}}
\newcommand{\SG}{\mathfrak{S}}
\newcommand{\bm}{\mathbf{m}}
\newcommand{\bp}{\mathbf{p}}
\newcommand{\bk}{\mathbf{k}}
\newcommand{\bg}{\mathbf{g}}

\newcommand{\Sym}{\mathrm{Sym}}

\def\Mper{M}
\def\upi{U}
\def\csupp{{\rm csupp}}

\newcommand{\larc}[1]{\hspace{-.4ex}\overset{#1}{\frown}\hspace{-.4ex}}
\newcommand{\slarc}[1]{\overset{#1}{\frown}}

\makeatletter
\renewcommand{\@makefnmark}{\mbox{\textsuperscript{}}}
\makeatother

\def\adots{\mathinner{\mkern2mu\raise0pt\hbox{.}  
\mkern2mu\raise4pt\hbox{.}\mkern1mu
\raise7pt\vbox{\kern7pt\hbox{.}}\mkern1mu}}

\allowdisplaybreaks[1]

\begin{document}

\title{Supercharacters, symmetric functions in noncommuting\\ variables, and related Hopf algebras}

\author{
 Marcelo Aguiar, Carlos Andr\'e, Carolina Benedetti, Nantel Bergeron,\\ Zhi Chen, Persi Diaconis, Anders Hendrickson, Samuel Hsiao, I. Martin Isaacs,\\ Andrea Jedwab, Kenneth Johnson, Gizem Karaali, Aaron Lauve, Tung Le,\\ Stephen Lewis, Huilan Li, Kay Magaard, Eric Marberg, Jean-Christophe Novelli,\\ Amy Pang, Franco Saliola, Lenny Tevlin, Jean-Yves Thibon, Nathaniel Thiem,\\ Vidya Venkateswaran, C. Ryan Vinroot, Ning Yan, Mike Zabrocki}

\date{}

\maketitle

\begin{abstract}
We identify two seemingly disparate structures: supercharacters, a useful way of doing Fourier analysis on the group of unipotent uppertriangular matrices with coefficients in a finite field, and the ring of symmetric functions in noncommuting variables.  Each is a Hopf algebra and the two are isomorphic as such.  This allows developments in each to be transferred.  The identification suggests a rich class of examples for the emerging field of combinatorial Hopf algebras.
\end{abstract}

\section{Introduction}

Identifying structures in seemingly disparate fields is a basic task of mathematics.  An example, with parallels to the present work, is the identification of the character theory of the symmetric group with symmetric function theory.  This connection is wonderfully exposited in Macdonald's book \cite{Mcd}.  Later, Geissinger and Zelevinsky independently realized that there was an underlying structure of Hopf algebras that forced and illuminated the identification \cite{Ge77,Ze}.  We present a similar program for a ``supercharacter'' theory associated to the uppertriangular group and the symmetric functions in noncommuting variables.  

\subsection{Uppertriangular matrices}

Let $\UT_n(q)$ be the group of uppertriangular matrices with entries in the finite field $\FF_q$ and ones on the diagonal.  This group is a Sylow $p$-subgroup of $\GL_n(q)$.  Describing the conjugacy classes or characters of $\UT_n(q)$ is a provably ``wild'' problem.  In a series of papers, Andr\'e developed a cruder theory that lumps together various conjugacy classes into ``superclasses'' and considers certain sums of irreducible characters as ``supercharacters."  The two structures are compatible (so supercharacters are constant on superclasses).  The resulting theory is very nicely behaved --- there is a rich combinatorics describing induction and restriction along with an elegant formula for the values of supercharacters on superclasses.  The combinatorics is described in terms of set partitions (the symmetric group theory involves integer partitions) and the combinatorics seems akin to tableau combinatorics.  At the same time, supercharacter theory is rich enough to serve as a substitute for ordinary character theory in some problems \cite{ADS} .

In more detail, the group $\UT_n(q)$ acts on both sides of the algebra of strictly upper-triangular matrices $\fkn_n$ (which can be thought of as $\fkn_n=\UT_n(q)-1$).  The two sided orbits on $\fkn_n$ can be mapped back to $\UT_n(q)$ by adding the identity matrix.  These orbits form the superclasses in $\UT_n(q)$.  A similar construction on the dual space $\fkn^*_n$ gives a collection of class functions on $\UT_n(q)$ that turn out to be constant on superclasses.  These orbit sums (suitably normalized) are the supercharacters.  Let 
$$\SC=\bigoplus_{n\geq 0} \SC_n,$$
where $\SC_n$ is the set of functions from $\UT_n(q)$ to $\CC$ that are constant on superclasses, and $\SC_0=\CC\spanning\{1\}$ is by convention the set of class functions of $\UT_0(q)=\{\}$.

It is useful to have a combinatorial description of the superclasses in $\SC_n$.   These are indexed by elements of $\fkn_n$ with at most one nonzero entry in each row and column.   Every superclass contains a unique such matrix, obtained by a set of elementary row and column operations.   Thus, when $n=3$, there are five such patterns; with $\ast\in \FF_q^\times$,
$$\left(\begin{array}{ccc}
0 & 0 & 0\\
0 & 0 & 0\\
0 & 0 & 0
\end{array}\right),\quad
\left(\begin{array}{ccc}
0 & \ast & 0\\
0 & 0 & 0\\
0 & 0 & 0
\end{array}\right),\quad
\left(\begin{array}{ccc}
0 & 0 & 0\\
0 & 0 & \ast\\
0 & 0 & 0
\end{array}\right),\quad
\left(\begin{array}{ccc}
0 & 0 & \ast\\
0 & 0 & 0\\
0 & 0 & 0
\end{array}\right),\quad \text{and}\quad
\left(\begin{array}{ccc}
0 & \ast & 0\\
0 & 0 & \ast\\
0 & 0 & 0
\end{array}\right).$$
Each representative matrix $X$ can be encoded as a pair $(D,\phi)$, where $D=\{(i,j)\mid X_{ij}\neq 0\}$ and $\phi:D\rightarrow \FF_q^\times$ is given by $\phi(i,j)=X_{ij}$.   There is a slight abuse of notation here since the pair $(D,\phi)$ does not record the size of the matrix $X$.  Let $X_{D,\phi}$ denote the distinguished representative corresponding to the pair $(D,\phi)$, and let  $\kappa_{D,\phi}=\kappa_{X_{D,\phi}}$ be the function that is 1 on the superclass and zero elsewhere.

We give combinatorial expressions for the product and coproduct
in this section and representation theoretic descriptions in Section \ref{PiHopfAlgebra}.  The product is given by
\begin{equation}\label{IntroProduct}
\kappa_{X_{D,\phi}}\cdot \kappa_{X_{D',\phi'}}=\sum_{X'}  \kappa_{\left(\begin{smallmatrix} X_{D,\phi} & X'\\ 0 & X_{D',\phi'}\end{smallmatrix}\right)},
\end{equation}
where the sum runs over all ways of placing a matrix $X'$  into the upper-right hand block such that the resulting matrix still has at most one nonzero entry in each row and column.   Note that this differs from the pointwise product of class functions,
which is internal to each $\SC_n$ (and hence does not turn $\SC$ into a
graded algebra).

For example, if
\begin{align*}
(D,\phi)&=(\{\},\phi)\leftrightarrow \left(\begin{array}{cc} 0 & 0 \\ 0 & 0\end{array}\right)\\
(D',\phi')&=(\{(1,2),(2,3)\},\{\phi(1,2)=a,\phi(2,3)=b\})\leftrightarrow \left(\begin{array}{ccc}0 & a & 0\\ 0 & 0 & b\\ 0 & 0& 0\end{array}\right),
\end{align*}
 where the sizes of the matrices are $2$ and $3$, respectively, then
$$\kappa_{D,\phi}\cdot \kappa_{D',\phi'} =\kappa_{\left(\begin{smallmatrix} 
0 & 0 & 0 & 0 & 0\\
0 & 0 & 0 & 0 &  0\\
0 &  0 & 0 & a & 0\\
0 & 0 &  0 & 0 & b\\
0 & 0 & 0 &  0 & 0\end{smallmatrix}\right)}+\sum_{c\in\FF_q^\times} \kappa_{\left(\begin{smallmatrix} 
0 & 0 & c & 0 & 0\\
0 & 0 & 0 & 0 &  0\\
0 &  0 & 0 & a & 0\\
0 & 0 &  0 & 0 & b\\
0 & 0 & 0 &  0 & 0\end{smallmatrix}\right)}+
\kappa_{\left(\begin{smallmatrix} 
0 & 0 & 0 & 0 & 0\\
0 & 0 & c & 0 &  0\\
0 &  0 & 0 & a & 0\\
0 & 0 &  0 & 0 & b\\
0 & 0 & 0 &  0 & 0\end{smallmatrix}\right)}.$$

We can define the coproduct on $\SC_n$ by
\begin{equation}\label{IntroCoProduct}
\Delta(\kappa_{X_{D,\phi}})=\sum_{{[n]=S\cup S^c\atop (i,j)\in D\text{ only if}}\atop \text{$i,j\in S$ or $i,j\in S^c$}}
 \kappa_{(X_{D,\phi})_S}\otimes \kappa_{(X_{D,\phi})_{S^c}},
 \end{equation}
where $(X)_S$ is the matrix restricted to the rows and columns in $S$. 
For example, if $D=\{(1,4),(2,3)\}$, $\phi(2,3)=a$, and $\phi(1,4)=b$, then
$$\Delta(\kappa_{D,\phi})=\kappa_{D,\phi}\otimes 1+ \kappa_{\left(\begin{smallmatrix} 
0 & a\\
0 & 0\end{smallmatrix}\right)}\otimes  \kappa_{\left(\begin{smallmatrix} 
0 & b\\
0 & 0\end{smallmatrix}\right)}+ \kappa_{\left(\begin{smallmatrix} 
0 & b\\
0 & 0\end{smallmatrix}\right)}\otimes  \kappa_{\left(\begin{smallmatrix} 
0 & a\\
0 & 0\end{smallmatrix}\right)}+ 1 \otimes \kappa_{D,\phi}.$$

In Section \ref{PiHopfAlgebra}, we show that the product and coproduct above have a representation theoretic meaning and we prove that

\medskip

\noindent\textbf{Corollary \ref{PiIsHopfAlgebra}} \emph{With the product (\ref{IntroProduct}) and the coproduct (\ref{IntroCoProduct}), the space $\SC$ forms a Hopf algebra.}

\medskip

Background on Hopf algebras is in Section \ref{BackgroundHopfAlgebras}.   We note here that $\SC$ is graded, noncommutative, and cocommutative.  It has a unit $\kappa_\emptyset\in \SC_0$ and a counit $\varepsilon:\SC\rightarrow \CC$ obtained by taking the coefficient of $\kappa_\emptyset$.

\subsection{Symmetric functions in noncommuting variables}

Let $\lambda$ be a set partition of $[n]=\{1,2,\ldots, n\}$, denoted $\lambda\vdash [n]$.  A monomial of shape $\lambda$ is a product of noncommuting variables $a_1 a_2 \cdots a_k$, where variables are equal if and only if the
corresponding indices/positions are in the same block/part of $\lambda$.  For example, if $135|24\vdash [5]$, then $xyxyx$ is a monomial of shape $\lambda$ ($135|24$ is the set partition of $[5]$ with parts $\{1,3,5\}$ and $\{2,4\}$).  Let $\bm_\lambda$ be the sum of all monomials of shape $\lambda$.  Thus, with three variables 
$$\bm_{135|24}=xyxyx+yxyxy+xzxzx+zxzxz+yzyzy+zyzyz.$$
Usually, we work with an infinite set of variables and formal sums.

Define 
$$\NCSym=\bigoplus_{n\geq 0} \NCSym_n,\qquad \text{where}\qquad \NCSym_n=\CC\spanning\{\bm_\lambda\mid \lambda\vdash[n]\}.$$
The elements of $\NCSym$ are called symmetric functions in noncommuting
variables. As linear combinations of the $\bm_\lambda$'s, they are invariant under permutations of variables.
Such functions were considered by Wolf  \cite{Wo} and Doubilet \cite{Do}.
More recent work of Sagan brought them to the forefront. A lucid
introduction is given by Rosas and Sagan \cite{SR} and combinatorial
applications by Gebhard and Sagan \cite{GS}.  The algebra $\NCSym$ is actively studied as part of the theory of combinatorial Hopf algebras \cite{AM,BHRZ,BRRZ,BZ,HNT,NT06}.  

\begin{remark}
There are a variety of notations given for $\NCSym$, including $\mathbf{NCSym}$ and $\mathbf{WSym}$.  Instead of choosing between these two conventions, we will use the more generic $\NCSym$, following Rosas and Sagan.
\end{remark}

Here is a brief definition of product and coproduct; Section \ref{BackgroundNCSym} has more details.  If $\lambda\vdash[k]$ and $\mu\vdash [n-k]$, then
\begin{equation}\label{mult_monomial}
\bm_\lambda \bm_\mu=\sum_{\nu\vdash [n]\atop \nu\wedge ([k]|[n-k])=\lambda\mid\mu} \bm_\nu.
\end{equation}
where $\wedge$ denotes the join in the poset of set partitions under refinement (in this poset $1234$ precedes the two incomparable set partitions $1|234$ and $123|4$), and $\lambda\mid\mu\vdash[n]$ is the set partition
\begin{equation}\label{SetPartitionConcatenation}
\lambda\mid \mu = \lambda_1|\lambda_2|\cdots|\lambda_a|\mu_1+k|\mu_2+k|\cdots|\mu_b+k.
\end{equation}
Thus, if $\lambda=1|2$ and $\mu=123$, then $\lambda \mid \mu=1|2|345$.

The coproduct is defined by
\begin{equation}\label{mcomult}
\Delta(\bm_\lambda)=\sum_{J\subseteq [\ell(\lambda)]} \bm_{\mathrm{st}(\lambda_J)}\otimes \bm_{\mathrm{st}(\lambda_{J^c})},
\end{equation}
where $\lambda$ has $\ell(\lambda)$ parts, $\lambda_J=\{\lambda_j\in \lambda\mid j\in J\}$,  $\st:J\rightarrow [|J|]$ is the unique order preserving bijection, and $J^c=[\ell(\lambda)]\setminus J$.

Thus,
\begin{align*}
\Delta(\bm_{14|2|3}) = &\bm_{14|2|3}\otimes 1 +2\bm_{13|2}\otimes \bm_1+\bm_{12}\otimes \bm_{1|2}+\bm_{1|2}\otimes \bm_{12}\\
&+2\bm_1\otimes \bm_{13|2}+1\otimes \bm_{14|2|3}.
\end{align*}

It is known (\cite[Section 6.2]{AM},\cite[Theorem 4.1]{BRRZ}) that $\NCSym$, endowed with this product and coproduct  is a Hopf algebra, where  the antipode is inherited from the grading.  A basic result of the present paper is stated here for $q=2$ (as described in Section \ref{BackgroundSupercharacters} below, the pairs $(D,\phi)$ are in correspondence with set-partitions).  The version for general $q$ is stated in Section~\ref{ColoredCorrespondence}.

\medskip

\noindent\textbf{Theorem \ref{q=2Correspondence}.}
\emph{For $q=2$, the function
$$\begin{array}{rccc} \ch: & \SC & \longrightarrow & \NCSym\\ 
& \kappa_\mu & \mapsto & \bm_\mu\end{array}$$
is a Hopf algebra isomorphism.}

\medskip

This construction of a Hopf algebra from the representation theory of
a sequence of groups
is the main contribution of this paper. It differs from previous work
in that supercharacters are used.
Previous work was confined to ordinary characters (e.g. \cite{Li}) and the
results of \cite{BLL} indicate that
this is a restrictive setting. This work opens the possibility for a
vast new source of Hopf algebras.

Section 2 gives further background on supercharacters (\ref{BackgroundSupercharacters}), some representation theoretic operations (\ref{BackgroundRepresentationTheory}), Hopf algebras (\ref{BackgroundHopfAlgebras}), and symmetric functions in noncommuting variables (\ref{BackgroundNCSym}).  Section 3 proves the isomorphism theorem for general $q$, and Section 4 proves an analogous realization for the dual Hopf algebra.   The appendix describes the available Sage programs developed in parallel with the present study, and a link for a list of open problems.

\subsubsection*{Acknowledgements}

This paper developed during a focused research week at the American Institute of Mathematics in May 2010.  The main results presented here were proved as a group during that meeting.

\section{Background}

\subsection{Supercharacter theory}\label{BackgroundSupercharacters}

Supercharacters were first studied by Andr\'e (e.g. \cite{An95}) and Yan \cite{Ya01} in relation to $\UT_n(q)$ in order to find a more tractable way to understand the representation theory of $\UT_n(q)$.  Diaconis and Isaacs \cite{DI08} then generalized the concept to arbitrary finite groups, and we reproduce a version of this more general definition below.  

A \emph{supercharacter theory} of a finite group $G$ is a pair $(\cK,\cX)$ where $\cK$ is a partition of $G$ and $\cX$ is a partition of the irreducible characters of $G$ such that 
\begin{enumerate}
\item[(a)] Each $K\in\cK$ is a union of conjugacy classes,
\item[(b)] $\{1\}\in\cK$, where $1$ is the identity element of $G$, and $\{\One\}\in\cX$, where $\One$ is the trivial character of $G$. 
\item[(c)] For $X\in \cX$, the character
$$\sum_{\psi\in X}\psi(1)\psi$$
is constant on the parts of $\cK$,
\item[(d)] $|\cK|=|\cX|$.
\end{enumerate}
We will refer to the parts of $\cK$ as \emph{superclasses}, and for some fixed choice of scalars $c_X\in \QQ$ (which are not uniquely determined), we will refer to the characters 
$$\chi^X=c_X\sum_{\psi\in X}\psi(1)\psi,\qquad \text{for  $X\in\cX$}$$
 as \emph{supercharacters} (the scalars $c_X$ should be picked such that the supercharacters are indeed characters).  For more information on the implications of these axioms, including some redundancies in the definition, see \cite{DI08}.

There are a number of different known ways to construct supercharacter theories for groups, including 
\begin{itemize}
\item Gluing together group elements and irreducible characters using outer automorphisms \cite{DI08},
\item Finding normal subgroups $N\triangleleft G$ and grafting together superchararacter theories for the normal subgroup $N$ and for the factor group $G/N$ to get a supercharacter theory for the whole group \cite{He08}.
\end{itemize}
This paper will however focus on a technique first introduced for algebra groups \cite{DI08}, and then generalized to some other types of groups by Andr\'e and Neto (e.g. \cite{AN06}). 

The group $\UT_n(q)$ has a natural two-sided action on the $\FF_q$-spaces
$$\fkn=\UT_n(q)-1\qquad\text{and}\qquad \fkn^*=\Hom(\fkn,\FF_q)$$
given by left and right multiplication on $\fkn$ and for $\lambda\in \fkn^*$,
$$(u\lambda v)(x-1)=\lambda(u^{-1}(x-1)v^{-1}), \qquad\text{for $u,v,x\in \UT_n(q)$}.$$
It can be shown that the orbits of these actions parametrize the superclasses and supercharacters, respectively, for a supercharacter theory.  In particular, two elements $u,v\in\UT_n(q)$ are in the same superclass if and only if $u-1$ and $v-1$ are in the same two-sided orbit in $\UT_n(q)\backslash \fkn/\UT_n(q)$.  Since the action of $\UT_n(q)$ on $\fkn$ can be viewed as applying row and column operations, we obtain a parameterization of superclasses given by
$$\left\{\begin{array}{c} \text{Superclasses}\\ \text{of $\UT_n(q)$}\end{array}\right\}\longleftrightarrow \left\{\begin{array}{c}\text{$u-1\in \fkn$ with at most}\\ \text{one nonzero entry in}\\ \text{each row and column}\end{array}\right\}.$$ 
This indexing set is central to the combinatorics of this paper, so we give several interpretations for it.  Let
\begin{align*}
\cM_n(q) & = \left\{(D,\phi)\ \bigg|\   \begin{array}{@{}l} D\subseteq \{(i,j)\mid 1\leq i<j\leq n\}, \phi:D\rightarrow \FF_q^\times,\\  (i,j),(k,l)\in D\text{ implies } i\neq k,j\neq l\end{array}\right\}\\
\cS_n(q) & =  \left\{\begin{array}{c}\text{Sets $\lambda$ of triples $i\larc{a}j=(i,j,a)\in [n]\times [n]\times \FF_q^\times$,}\\ \text{with $i<j$, and $i\larc{a}j,k\larc{b}l\in \lambda$ implies $i\neq k,j\neq l$}\end{array}\right\},
\end{align*}
where we will refer to the elements of $\cS_n(q)$ as \emph{$\FF_q^\times$-set partitions}.  In particular,
\begin{equation}\label{CombinatorialIndexSets}
 \begin{array}{ccccc}
\cM_n(q) & \longleftrightarrow & \cS_n(q) & \longleftrightarrow & \left\{\begin{array}{c}\text{$u-1\in \fkn$ with at most}\\ \text{one nonzero entry in}\\ \text{each row and column}\end{array}\right\}\\
(D,\phi) & \mapsto & \lambda=\{i\larc{\phi(i,j)}j\mid (i,j)\in D\} & \mapsto &  \displaystyle \sum_{i\slarc{a}j\in \lambda} ae_{ij},
\end{array}
\end{equation}
where $e_{ij}$ is the matrix with 1 in the $(i,j)$ position and zeroes elsewhere.  The following table lists the correspondences for $n=3$.
$$\begin{array}{|c|c|c|c|c|c|} \hline 
\text{Superclass} &
 \left(\begin{array}{@{}c@{\ }c@{\ }c@{}}
0 & 0 & 0\vspace{-.1cm}\\
0 & 0 & 0\vspace{-.1cm}\\
0 & 0 & 0
\end{array}\right)&
\left(\begin{array}{@{}c@{\ }c@{\ }c@{}}
0 & a & 0\vspace{-.1cm}\\
0 & 0 & 0\vspace{-.1cm}\\
0 & 0 & 0
\end{array}\right)&
\left(\begin{array}{@{}c@{\ }c@{\ }c@{}}
0 & 0 & 0\vspace{-.1cm}\\
0 & 0 & a\vspace{-.1cm}\\
0 & 0 & 0
\end{array}\right)&
\left(\begin{array}{@{}c@{\ }c@{\ }c@{}}
0 & 0 & a\vspace{-.1cm}\\
0 & 0 & 0\vspace{-.1cm}\\
0 & 0 & 0
\end{array}\right)&
\left(\begin{array}{@{}c@{\ }c@{\ }c@{}}
0 & a & 0\vspace{-.1cm}\\
0 & 0 & b\vspace{-.1cm}\\
0 & 0 & 0
\end{array}\right)\\ \hline
\cM_3(q) & D=\{\} & \begin{array}{@{}c@{}} D=\{(1,2)\},\\ \phi(1,2)=a\end{array} &  \begin{array}{@{}c@{}} D=\{(2,3)\},\\ \phi(2,3)=a\end{array}  &  \begin{array}{@{}c@{}} D=\{(1,3)\},\\ \phi(1,3)=a\end{array}  &  \begin{array}{@{}c@{}} D=\{(1,2),(2,3)\},\\ \phi(1,2)=a, \phi(2,3)=b\end{array} \\ \hline
\cS_3(q) &
\begin{tikzpicture}[baseline=.2cm]
	\foreach \x in {1,2,3} 
		\node (\x) at (\x/2,0) [inner sep=0pt] {$\bullet$};
	\foreach \x in {1,2,3} 
		\node at (\x/2,-.2) {$\scs\x$}; 
\end{tikzpicture} 
& 
\begin{tikzpicture}[baseline=.2cm]
	\foreach \x in {1,2,3} 
		\node (\x) at (\x/2,0) [inner sep=0pt] {$\bullet$};
	\foreach \x in {1,2,3} 
		\node at (\x/2,-.2) {$\scs\x$};
	\draw (1) .. controls (1.25/2,.5) and (1.75/2,.5) ..  node [above=-2pt] {$\scs a$} (2); 
\end{tikzpicture} 
&
\begin{tikzpicture}[baseline=.2cm]
	\foreach \x in {1,2,3} 
		\node (\x) at (\x/2,0) [inner sep=0pt] {$\bullet$};
	\foreach \x in {1,2,3} 
		\node at (\x/2,-.2) {$\scs\x$};
	\draw (2) .. controls (2.25/2,.5) and (2.75/2,.5) ..  node [above=-2pt] {$\scs a$} (3); 
\end{tikzpicture} 
&
\begin{tikzpicture}[baseline=.2cm]
	\foreach \x in {1,2,3} 
		\node (\x) at (\x/2,0) [inner sep=0pt] {$\bullet$};
	\foreach \x in {1,2,3} 
		\node at (\x/2,-.2) {$\scs\x$};
	\draw (1) .. controls (1.5/2,.75) and (2.5/2,.75) ..  node [above=-2pt] {$\scs a$} (3); 
\end{tikzpicture} 
&
\begin{tikzpicture}[baseline=.2cm]
	\foreach \x in {1,2,3} 
		\node (\x) at (\x/2,0) [inner sep=0pt] {$\bullet$};
	\foreach \x in {1,2,3} 
		\node at (\x/2,-.2) {$\scs\x$};
	\draw (1) .. controls (1.25/2,.5) and (1.75/2,.5) ..  node [above=-2pt] {$\scs a$} (2); 
	\draw (2) .. controls (2.25/2,.5) and (2.75/2,.5) ..  node [above=-2pt] {$\scs b$} (3); 
\end{tikzpicture} \\ \hline \end{array}$$

\begin{remark}
Consider the maps
\begin{equation}\label{UnderlyingSetPartition}
\begin{array}{rccc} \pi:&\cM_n(q)& \longrightarrow &\cM_n(2)\\
&(D,\phi) & \mapsto & (D,1)\end{array} \qquad \text{and}\qquad  \begin{array}{rccc} \pi:&\cS_n(q) & \rightarrow & \cS_n(2)\\ &\lambda & \mapsto & \{i\larc{1}j\mid i\larc{a}j\in \lambda\}.\end{array}  
\end{equation}
They ignore the part of the data that involves field scalars.
Note that $\cM_n(2)$ and $\cS_n(2)$ are in bijection with the set of partitions
of the set $\{1, 2,\ldots, n\}$. Indeed, the connected components of an
element $\lambda \in \cS_n(2)$ may be viewed as the blocks of a partition
of $\{1, 2,\ldots, n\}$. Composing the map $\pi$ with this bijection
associates a set partition to an element of $\cM_n(q)$ or $\cS_n(q)$, which we call
the \emph{underlying} set partition.
\end{remark}

Fix a nontrivial homomorphism $\vartheta:\FF_q^+\rightarrow \CC^\times$.  For each $\lambda\in \fkn^*$, construct a $\UT_n(q)$-module 
$$V^\lambda=\CC\spanning\{v_\mu\mid \mu\in -\UT_n(q)\cdot\lambda\}$$
with left action given by
$$uv_\mu=\vartheta\big(\mu(u^{-1}-1)\big)v_{u\mu},\qquad\text{for $u\in \UT_n(q)$, $\mu\in  -\UT_n(q)\lambda$}.$$
It turns out that, up to isomorphism, these modules depend only on the two-sided orbit in $\UT_n(q)\backslash \fkn^*/\UT_n(q)$ of $\lambda$.    Furthermore, there is an injective function $\iota:  \cS_n(q)  \rightarrow  \fkn^*$ given by
$$\begin{array}{rccc} \iota(\lambda):& \fkn & \longrightarrow & \FF_q\\ & X & \mapsto & \dd\sum_{i\slarc{a}j\in \lambda} aX_{ij}\end{array}$$
that maps $\cS_n(q)$ onto a natural set of orbit representatives in $\fkn^*$.  We will identify $\lambda\in \cS_n(q)$ with $\iota(\lambda)\in \fkn^*$.

The traces of the modules $V^\lambda$ for $\lambda\in \cS_n(q)$ are the supercharacters of $\UT_n(q)$, and they have a nice supercharacter formula  given by
\begin{equation}\label{SupercharacterFormula}
\chi^\lambda(u_\mu)=\left\{\begin{array}{ll} \displaystyle\frac{q^{\#\{(i,j,k)\mid i<j<k,i\slarc{a}k\in \lambda\}}}{q^{\#\{(i\slarc{a}l,j\slarc{b}k)\in \lambda\times \mu\mid i<j<k<l\}}} \prod_{i\slarc{a}l\in \lambda\atop i\slarc{b}l\in \mu}\vartheta(ab), & \begin{array}{@{}l@{}}\text{if $i\larc{a}k\in \lambda$ and $i<j<k$}\\ \text{implies $i\larc{b}j,j\larc{b}k\notin \mu$,}\end{array}\\ 0, & \text{otherwise,}\end{array}\right.
\end{equation}
where $u_\mu$ has superclass type $\mu$ \cite{ADS}.  Note that the degree of the supercharacter is
\begin{equation}\label{SupercharacterDegree}
\chi^\lambda(1)=\prod_{i\slarc{a}k\in \lambda} q^{k-i-1}.
\end{equation}

Define 
$$\SC=\bigoplus_{n\geq 0} \SC_n,\qquad\text{where}\qquad \SC_n=\CC\spanning\{\chi^\lambda\mid \lambda\in \cS_n(q)\},$$
and let $\SC_0=\CC\spanning\{\chi^\emptyset\}$.  By convention, we write $1=\chi^{\emptyset}$, since this element will be the identity of our Hopf algebra.
Note that since $\SC_n$ is in fact the space of superclass functions of $\UT_n(q)$, it also has another distinguished basis, the superclass characteristic functions, 
$$\SC_n=\CC\spanning\{\kappa_\mu\mid \mu\in \cS_n(q)\},\qquad \text{where}\qquad \kappa_\mu(u)=\left\{\begin{array}{ll} 1, &\text{if $u$ has superclass type $\mu$,}\\ 0, &\text{otherwise,}\end{array}\right.$$
and $\kappa_\emptyset=\chi^\emptyset$.  Section 3 will explore a Hopf structure for this space.

We conclude this section by remarking that with respect to the usual inner product on class functions 
$$\langle \chi,\psi\rangle =\frac{1}{|\UT_n(q)|}\sum_{u\in \UT_n(q)} \chi(u)\overline{\psi(u)}$$
the supercharacters are orthogonal.  In fact, for $\lambda,\mu\in \cS_n(q)$,
\begin{equation}\label{InnerProductSupercharacters}
\langle \chi^\lambda,\chi^\mu\rangle=\delta_{\lambda\mu} q^{C(\lambda)}, \qquad \text{where } C(\lambda)=\#\{(i,j,k,l)\mid i\larc{a}k,j\larc{b}l\in \lambda\}.\end{equation}
In particular, this inner product remains nondegenerate on $\SC_n$.

\subsection{Representation theoretic functors on $\SC$} \label{BackgroundRepresentationTheory}

We will focus on a number of representation theoretic operations on the space $\SC$.   For $J=(J_1|J_2| \cdots|J_\ell)$ any set composition of $\{1,2,\ldots, n\}$, let
\begin{equation*}
\UT_J(q)=\{u\in \UT_n(q)\mid u_{ij}\neq 0 \text{ with  $i<j$ implies $i,j$ are in the same part of $J$}\}.
\end{equation*}
In the remainder of the paper we will need variants of a straightening map on set compositions.  For each set composition $J=(J_1| J_2| \cdots| J_\ell)$, let 
\begin{equation}\label{StraighteningSetPartitions}
\st_J([n])=\st_{J_1}(J_1)\times \st_{J_2}(J_2)\times \cdots\times \st_{J_\ell}(J_\ell),\end{equation}
where for $K\subseteq [n]$, $\st_K:K\longrightarrow [|K|]$ is the unique order preserving map.
For example,\\ $\st_{(14|3|256)}([6])=\{1,2\}\times \{1\}\times\{1,2,3\}$. 
 
We can extend this straightening map to a canonical isomorphism
\begin{equation} \label{GroupStraightening}
\st_J:\UT_J(q)\longrightarrow \UT_{|J_1|}(q)\times \UT_{|J_2|}(q)\times \cdots \times \UT_{|J_\ell|}(q)
\end{equation}
by reordering the rows and columns according to (\ref{StraighteningSetPartitions}).
For example, if $J=(14|3|256)$, then
$$\UT_J(q)\ni\left(\begin{array}{@{\ }c@{\ }c@{\ }c@{\ }c@{\ }c@{\ }c@{\ }} 
1 & 0 & 0 & a & 0 & 0 \vspace{-.1cm}\\
0 & 1 & 0 & 0 & b & c\vspace{-.1cm}\\
0 & 0 & 1 & 0 & 0 & 0\vspace{-.1cm}\\
0 & 0 & 0 & 1 & 0 & 0\vspace{-.1cm}\\
0 & 0 & 0 & 0 & 1 & d\vspace{-.1cm}\\
0 & 0 & 0 & 0 & 0 & 1\end{array}\right) \overset{\st_J}{\longmapsto} \left(\left(\begin{array}{@{\ }c@{\ }c@{\ }} 
1 &  a\vspace{-.1cm}\\
0 & 1 \end{array}\right),(1),\left(\begin{array}{@{\ }c@{\ }c@{\ }c@{\ }} 
1 &  b & c\vspace{-.1cm}\\
0 & 1 & d\vspace{-.1cm}\\
0 & 0 & 1\end{array}\right)\right)\in \UT_2(q)\times\UT_1(q)\times  \UT_3(q) .$$
Combinatorially, if  $J=(J_1|J_2|\cdots|J_\ell)$ we let
$$\cS_J(q)=\{\lambda\in \cS_n(q)\mid i\larc{a}j\in \lambda\text{ implies $i,j$ are in the same part in } J\}.$$
Then we obtain the bijection
\begin{equation}\label{StraighteningLabeledSetPartitions}
\mathrm{st}_J: \cS_J(q) \longrightarrow    \cS_{|J_1|}(q)\times \cS_{|J_2|}(q)\times \cdots \times \cS_{|J_\ell|}(q)
\end{equation}
that relabels the indices using the straightening map (\ref{StraighteningSetPartitions}).  For example, if $J=14|3|256$, then
$$
\st_J\bigg(\begin{tikzpicture}[baseline=.2cm]
	\foreach \x in {1,...,6} 
		\node (\x) at (\x/2,0) [inner sep=0pt] {$\bullet$};
	\foreach \x in {1,...,6} 
		\node at (\x/2,-.25) {$\scs\x$};
	\draw (1) .. controls (2/2,.75) and (3/2,.75) ..  node  [above] {$\scs a$} (4); 
	\draw (2) .. controls (3.25/2,1) and (4.75/2,1) .. node [above] {$\scs b$} (6);
\end{tikzpicture}\bigg)
= 
\begin{tikzpicture}[baseline=.2cm]
	\foreach \x in {1,2} 
		\node (\x) at (\x/2,0) [inner sep=0pt] {$\bullet$};
	\foreach \x in {1,2} 
		\node at (\x/2,-.25) {$\scs\x$};
	\draw (1) .. controls (1.25/2,.25) and (1.75/2,.25) ..  node  [above] {$\scs a$} (2); 
\end{tikzpicture}
\times \begin{tikzpicture}[baseline=.2cm]
	\foreach \x in {1} 
		\node (\x) at (\x/2,0) [inner sep=0pt] {$\bullet$};
	\foreach \x in {1} 
		\node at (\x/2,-.25) {$\scs\x$};
\end{tikzpicture}
\times
\begin{tikzpicture}[baseline=.2cm]
	\foreach \x in {1,...,3} 
		\node (\x) at (\x/2,0) [inner sep=0pt] {$\bullet$};
	\foreach \x in {1,...,3} 
		\node at (\x/2,-.25) {$\scs\x$};
	\draw (1) .. controls (1.5/2,.5) and (2.5/2,.5) ..  node  [above] {$\scs b$} (3); 
\end{tikzpicture}$$
Note that $\UT_m(q)\times \UT_n(q)$ is an algebra group, so it has a supercharacter theory with the standard construction \cite{DI08} such that 
$$\SC(\UT_m(q)\times\UT_n(q))\cong \SC_m\otimes \SC_n.$$
The combinatorial map (\ref{StraighteningLabeledSetPartitions}) preserves supercharacters across this isomorphism.

The first two operations of interest are restriction
$$ \begin{array}{rccc} {}^J\Res^{\UT_n(q)}_{\st_J(\UT_J(q))}: &   \SC_n &   \longrightarrow & \SC_{|J_1|}\otimes \SC_{|J_2|}\otimes\cdots \otimes \SC_{|J_\ell|}\\
& \chi & \mapsto & \Res^{\UT_n(q)}_{\UT_J(q)}(\chi)\circ \st_J^{-1},
\end{array}$$
or
$$ {}^J\Res^{\UT_n(q)}_{\st_J(\UT_J(q))}(\chi)(u)= \chi(\st_J^{-1}(u)), \qquad \text{for $u\in \UT_{|J_1|}(q)\times\cdots\times\UT_{|J_\ell|}(q)$},$$
and its Frobenius adjoint map superinduction
$$\begin{array}{rccc} {}^J\SInd^{\UT_n(q)}_{\st_J(\UT_J(q))}:  &  \SC_{|J_1|}\otimes \SC_{|J_2|}\otimes\cdots \otimes \SC_{|J_\ell|} &\longrightarrow  &\SC_n\\
&\chi & \mapsto & \SInd^{\UT_n(q)}_{\UT_J(q)}(\st_J^{-1}(\chi)),\end{array}$$
where for a superclass function $\chi$ of $\UT_{J}(q)$,
$$\SInd^{\UT_n(q)}_{\UT_J(q)}(\chi)(u) = \frac{1}{|\UT_J(q)|^2}\sum_{x,y\in \UT_n(q)\atop x(u-1)y+1\in \UT_J(q)} \chi(x(u-1)y+1),\qquad\text{for $u\in \UT_n(q)$}.$$
Note that under the usual inner product on characters,
$$\left\langle \SInd^{\UT_n(q)}_{\UT_J(q)}(\psi), \chi\right\rangle=\left\langle \psi, \Res^{\UT_n(q)}_{\UT_J(q)}(\chi)\right\rangle.$$

\begin{remarks}\hfill

\begin{enumerate}
\item[(a)] While superinduction takes superclass functions to superclass functions, a superinduced character may not be the trace of a representation.   Therefore,  $\SInd$ is not really a functor on the module level.  An exploration of the relationship between superinduction and induction can be found in \cite{MT09}.
\item[(b)]  There is an algorithmic method for computing restrictions of supercharacters (and also tensor products of characters) \cite{Th,TV09}.   This has been implemented in Sage (see the Appendix, below).
\end{enumerate}
\end{remarks}

For an integer composition $(m_1,m_2,\ldots, m_\ell)$ of $n$, let
$$\UT_{(m_1,m_2,\ldots, m_\ell)}(q)=\UT_{(1,\ldots,m_1\mid m_1+1,\ldots, m_1+m_2\mid \cdots | n-m_\ell+1,\ldots,n)}(q) \subseteq \UT_{m_1+\cdots +m_\ell}(q).$$
There is a surjective homomorphism $\tau:\UT_n(q)\rightarrow \UT_{(m_1,m_2,\ldots, m_\ell)}(q)$ such that $\tau^2=\tau$ ($\tau$ fixes the subgroup $\UT_{(m_1,m_2,\ldots, m_\ell)}(q)$ and sends the normal complement to 1).  The next two operations arise naturally from this situation.  We have inflation
$$\Inf^{\UT_n(q)}_{\UT_{(m_1,m_2,\ldots, m_\ell)}(q)}: \SC_{m_1}\otimes \SC_{m_2}\otimes\cdots \otimes \SC_{m_\ell} \longrightarrow \SC_n,$$
where
$$\Inf^{\UT_n(q)}_{\UT_{(m_1,m_2,\ldots, m_\ell)}(q)}(\chi)(u)=\chi(\tau(u)), \qquad\text{for $u\in \UT_n(q)$},$$
and its Frobenius adjoint map deflation
$$\Def^{\UT_n(q)}_{\UT_{(m_1,m_2,\ldots, m_\ell)}(q)}: \SC_n\longrightarrow \SC_{m_1}\otimes \SC_{m_2}\otimes\cdots \otimes \SC_{m_\ell}, $$
where
$$\Def^{\UT_n(q)}_{\UT_{(m_1,m_2,\ldots, m_\ell)}(q)}(\chi)(u)=\frac{1}{|\ker(\tau)|}\sum_{v\in \tau^{-1}(u)} \chi(v), \qquad \text{for $u\in \UT_{(m_1,m_2,\ldots, m_\ell)}(q)$}.$$
On supercharacters, the inflation map is particularly nice, and is given combinatorially by
$$\Inf^{\UT_n(q)}_{\UT_{(m_1,m_2,\ldots, m_\ell)}(q)}(\chi^{\lambda_1}\times \chi^{\lambda_2}\times\cdots \times \chi^{\lambda_\ell})=\chi^{\lambda_1\mid \lambda_2 \mid \cdots \mid \lambda_\ell},$$
where $\lambda_1 \mid \lambda_2 \mid \cdots \mid \lambda_\ell$ is as in (\ref{SetPartitionConcatenation}) (see for example \cite{Th}).

\subsection{Hopf algebra basics}\label{BackgroundHopfAlgebras}

Hopf algebras arise naturally in combinatorics and algebra, where there are ``things" that break into parts that can also be put together with some compatibility between operations \cite{JR79}.  They have emerged as a central object of study in algebra through quantum groups \cite{CP94,Dr, SS93} and in combinatorics \cite{ABS06, AM10, HNT}.  Hopf algebras find applications in diverse fields  such as algebraic topology,
representation theory, and mathematical physics.

We suggest the first few chapters of \cite{SS93} for a motivated introduction and \cite{Mo93} as a  basic text.  Each has extensive references.  The present section gives definitions to make our exposition self-contained.

Let $\cA$ be an associative algebra with unit $1$ over a field $\KK$.  The unit can be associated with a map
$$\begin{array}{rccc} u: & \KK & \longrightarrow & \cA\\ & t & \mapsto & t \cdot 1\end{array}$$
A \emph{coalgebra} is a vector space $C$ over $\KK$ with two $\KK$-linear maps: the coproduct $\Delta:C\rightarrow C\otimes C$ and a counit $\veps:C\rightarrow \KK$.  The coproduct must be coassociative (as a map from $C$ to $C\otimes C\otimes C$), so that $(\Delta\otimes \Id)\circ \Delta=(\Id\otimes\Delta)\circ \Delta,$ or for $a\in C$,
$$\sum_{j} \Delta(b_j)\otimes c_j=\sum_j b_j\otimes \Delta(c_j), \qquad \text{if} \qquad \Delta(a)=\sum_{j} b_j\otimes c_j.$$
The counit must be compatible with the coproduct, so that $(\veps\otimes \Id)\circ\Delta=(\Id\otimes\veps)\circ\Delta=\Id,$ where we identify $C$ with $\KK\otimes C$ and $C\otimes \KK$.  More explicitly,
$$a=\sum_j \veps(b_j) c_j=\sum_j b_j\veps(c_j), \qquad \text{if} \qquad \Delta(a)=\sum_j b_j\otimes c_j.$$
  A map $\vphi:C\rightarrow D$ between coalgebras is a \emph{coalgebra map} if $\Delta_D\circ\vphi=(\vphi\otimes\vphi)\circ \Delta_C$, where $\Delta_C$ and $\Delta_D$ are the coproducts of $C$ and $D$, respectively. A subspace $I\subseteq C$ is a \emph{coideal} if $\Delta_C(I)\subseteq I\otimes C + C\otimes I$ and $\veps(I)=0$.  In this case, the quotient space $C/I$ is a coalgebra.

An algebra that is also a coalgebra is a \emph{bialgebra} if the operations are compatible: for the coproduct and product,
$$\Delta(xy)=\Delta(x)\Delta(y)\qquad \text{where}\qquad (a\otimes b)(c\otimes d)=ac\otimes bd,$$
for the counit and product,
$$\veps(xy)=\veps(x)\veps(y),$$ 
for the unit and coproduct
$$\Delta\circ u=(u\otimes u)\circ \delta,\qquad \text{where}\qquad \begin{array}{rccc} \delta: & \KK & \longrightarrow & \KK\otimes \KK\\ & t & \mapsto & t\otimes t,\end{array}$$
and for the counit and unit,
$$\veps\circ u=\Id.$$

For example, the group algebra $\KK[G]$ becomes a bialgebra under the maps $\Delta(g)=g\otimes g$ and $\veps(g)=1$ for all $g\in G$, and the polynomial algebra $\KK[x_1,\ldots, x_n]$ becomes a bialgebra under the operations $\Delta(x_i)=x_i\otimes 1 + 1\otimes x_i$ and $\veps(x_i)=0$ for $1\leq i\leq n$.  

A bialgebra is \emph{graded} if there is a direct sum decomposition 
$$\cA=\bigoplus_{n\geq 0} \cA_n,$$
such that $\cA_i\cA_j\subseteq \cA_{i+j}$, $u(\KK)\subseteq \cA_0$, $\Delta(\cA_n)\subseteq \bigoplus_{j=0}^n\cA_j\otimes \cA_{n-j}$ and $\veps(\cA_n)=0$ for all $n\geq 1$.  It is connected if $\cA_0\cong \KK$.  For example, the polynomial algebra is graded by polynomial degree.  In a bialgebra, an ideal that is also a coideal is called a \emph{biideal}, and the quotient is a bialgebra.

A Hopf algebra is a bialgebra with an antipode.  This is a linear map $S:\cA\rightarrow \cA$ such that if $\Delta(a)=\sum_k a_{k}\otimes a_k'$, then
\begin{equation}\label{AntipodeCondition}
\sum_k a_k S(a_k')=\veps(a)\cdot 1=\sum_k S(a_k)a_k'.
\end{equation}
For example, the bialgebra $\KK[G]$ has antipode $S(g)=g^{-1}$ and the bialgebra $\KK[x_1,\ldots,x_n]$ has antipode $S(x_i)=-x_i$.  More generally, if $\cA$ is a connected, graded bialgebra, then (\ref{AntipodeCondition}) can be solved inductively to give $S(t\cdot 1)=t\cdot 1$ for $t\cdot 1 \in \cA_0$, and for $a\in \cA_n$,
\begin{equation}\label{AntipodeRecursion}
S(a)=-a-\sum_{j=1}^{n-1} S(a_j)a_{n-j}',\qquad \text{where}\qquad \Delta(a)=a\otimes 1+1\otimes a+\sum_{j=1}^{n-1} a_j\otimes a_{n-j}'.
\end{equation}
Thus, any graded, connected bialgebra has an antipode and is automatically a Hopf algebra.

If $\mathcal{A}$ is a graded bialgebra (Hopf algebra) and each
$\mathcal{A}_n$ is finite-dimensional, then the graded dual
\[
\mathcal{A} = \bigoplus_{n\geq 0} \mathcal{A}_n^*
\]
is also a bialgebra (Hopf algebra). If $\mathcal{A}$ is commutative
(cocommutative), then
$\mathcal{A}^*$ is cocommutative (commutative).

\subsection{The Hopf algebra $\WSym$} \label{BackgroundNCSym}

Symmetric polynomials in a set of commuting variables $X$ are the
invariants  of the action of the symmetric group 
$\SG_X$ of $X$ by automorphisms of the polynomial algebra $\KK[X]$ over a field $\KK$.

When $X=\{x_1,x_2,\ldots\}$ is infinite, we let $\SG_X$ be the set of bijections on $X$ with finitely many nonfixed points.  Then the subspace of $\KK[[X]]^{\SG_X}$ of formal power series with bounded degree is the algebra
of symmetric functions $\mathrm{Sym}(X)$ over $\KK$. It has a natural bialgebra structure
defined by
\begin{equation}
\Delta(f) =  \sum_k f_k'\otimes f_k'',
\end{equation}
where the $f_k', f_k''$ are defined by the identity
\begin{equation}
f(X'+X'')   = \sum_k f_k'(X')f_k''(X''),
\end{equation}
and $X'+X''$ denotes the disjoint union of two copies of $X$. 
The advantage of defining the coproduct in this way is that $\Delta$  is clearly coassociative and
that it is obviously a morphism for the product. 
For each \emph{integer} partition $\lambda=(\lambda_1,\lambda_2,\ldots, \lambda_\ell)$, the monomial symmetric function corresponding to $\lambda$ is the sum
\begin{equation}
m_\lambda(X)=\sum_{x^\alpha\in O(x^\lambda)} x^\alpha
\end{equation}
over elements of the orbit $O(x^\lambda)$ of $x^\lambda=x_1^{\lambda_1}x_2^{\lambda_2}\cdots x_{\ell}^{\lambda_{\ell}}$ under $\SG_X$, and the monomial symmetric functions form a basis of $\mathrm{Sym}(X)$. The coproduct of a monomial function is 
\begin{equation}
\Delta (m_\lambda) = \sum_{\mu\cup\nu=\lambda}m_\mu\otimes m_\nu\,.
\end{equation}
The dual basis $m^*_\lambda$ of $m_\lambda$ is a multiplicative basis of the
graded dual $\Sym^*$, which turns out to be isomorphic to $\Sym$ via the identification
$m_n^*=h_n$ (the complete homogeneous function, the sum of all monomials of degree $n$).

The case of noncommuting variables is very similar. Let $A$ be
an alphabet, and consider the invariants of $\SG_A$ acting by automorphisms
on the free algebra $\KK\langle A\rangle$. Two words $a=a_1a_2\cdots a_n$ and $b=b_1b_2\cdots b_n$
are in the same orbit whenever $a_i=a_j$ if and only if $b_i=b_j$. Thus, orbits
are parametrized by set partitions  in at most $|A|$ blocks. Assuming as above
that $A$ is infinite, we obtain an algebra based on all set partitions,
defining the monomial basis by
\begin{equation}
\bm_\lambda(A)=\sum_{w\in O_\lambda}w,
\end{equation}
where $O_\lambda$ is the set of words such that $w_i=w_j$ if and only if $i$ and $j$ are
in the same block of $\lambda$.

One can introduce a bialgebra structure by means of the coproduct
\begin{equation}
\Delta(f) =  \sum_k f_k'\otimes f_k''    \qquad \text{where}\qquad           f(A'+A'')   = \sum_k f_k'(A')f_k''(A''),
\end{equation}
and $A'+A''$ denotes the disjoint union of two mutually commuting copies of
$A$. The coproduct of a monomial function is 
\begin{equation}\label{MonomialCoproduct}
\Delta(\bm_\lambda) = \sum_{J\subseteq [\ell(\lambda)]} \bm_{\mathrm{st}(\lambda_J)}\otimes \bm_{\mathrm{st}(\lambda_{J^c})}.
\end{equation}
This coproduct is  cocommutative. With the unit that sends $1$ to $m_\emptyset$ and the counit $\varepsilon( f(A))=f(0,0,\ldots)$, we have that $\NCSym$ is a connected graded bialgebra and therefore a graded Hopf algebra.

\begin{remark}
We again note that $\NCSym$ is often denoted in the literature as $\mathbf{NCSym}$ or $\mathbf{WSym}$. 
\end{remark}

\section{A Hopf algebra realization of $\SC$}\label{PiHopfAlgebra}

This section explicitly defines the Hopf structure on $\SC$ from a representation theoretic point of view.  We then work out the combinatorial consequences of these rules, and it directly follows that $\SC\cong \NCSym$ for $q=2$.  We then proceed to yield a ``colored" version of $\NCSym$ that will give the corresponding Hopf structure for the other values of $q$.

\subsection{The correspondence between $\SC$ and $\NCSym$}

In this section we describe a Hopf structure for the space 
\begin{align*}
\SC &= \bigoplus_{n\geq 0} \SC_n\\ 
&=\CC\spanning\{\kappa_\mu \mid \mu \in \cS_n(q),n\in \ZZ_{\geq 0}\}\\
&=\CC\spanning\{\chi^\lambda \mid \lambda \in \cS_n(q),n\in \ZZ_{\geq 0}\}.
\end{align*}
The product and coproduct are defined representation theoretically by the inflation and restriction operations of Section \ref{BackgroundRepresentationTheory},
\begin{equation}\label{InflationProduct}
\chi \cdot \psi=\Inf_{\UT_{(a,b)}(q)}^{\UT_{a+b}(q)}(\chi \times \psi), \qquad \text{where  $\chi\in \SC_a, \psi\in \SC_b$},
\end{equation}
and
\begin{equation}\label{RestrictionCoProduct}
\Delta(\chi)=\sum_{J=(A|A^c)\atop A\subset [n]} {}^{J}\Res_{\UT_{|A|}(q)\times \UT_{|A^c|}(q)}^{\UT_n(q)}(\chi), \qquad \text{for $\chi\in \SC_n$}.
\end{equation}
For a combinatorial description of the Hopf structure of $\SC$ it is most convenient to work with the superclass characteristic functions.  A matrix description appears in (\ref{IntroProduct}) and (\ref{IntroCoProduct}).
\begin{proposition} \label{CombinatorialSuperclassFunctions}\hfill

\begin{enumerate}
\item[(a)] For  $\mu\in \cS_k(q)$, $\nu\in \cS_{n-k}(q)$,
\begin{equation*}
\kappa_\mu\cdot \kappa_\nu=\sum_{\lambda=\mu\sqcup \gamma\sqcup (k+\nu)\in \cS_n(q)\atop   i\slarc{\scscs a}l\in \gamma \text{ implies } i\leq k < l}\kappa_\lambda,
\end{equation*}
where $(k+\nu)=\{(k+i)\larc{a}(k+j)\mid i\larc{a}j\in \nu\}$ and $\sqcup$ denotes disjoint union.
\item[(b)] For $\lambda\in \cS_n(q)$,
\begin{equation*}
\Delta(\kappa_\lambda)=\sum_{{\lambda=\mu\sqcup \nu\atop \mu\in \cS_A(q), \nu\in \cS_{A^c}(q)}\atop A\subseteq \{1,2,\ldots,n\}} \kappa_{\mathrm{st}_A(\mu)}\otimes \kappa_{\mathrm{st}_{A^c}(\nu)}.
\end{equation*}
\end{enumerate}
\end{proposition}\label{SuperclassFunctionsOperations}
\begin{proof}
(a) Let $(u_\lambda-1)\in \fkn_n$ be the natural orbit representative for the superclass corresponding to $\lambda$.  Then 
$$\Inf_{\UT_k(q)\times\UT_{n-k}(q)}^{\UT_n(q)}(\kappa_\mu\times\kappa_\nu)(u_\lambda)=(\kappa_\mu\times\kappa_\nu)(\tau(u_\lambda)),$$
where
$$\begin{array}{r@{}ccc}\tau: & \UT_n(q) & \longrightarrow & \UT_k(q)\times\UT_{n-k}(q)\\  & \left(\begin{array}{c|c} A & C\\  \hline 0 & B\end{array}\right) & \mapsto & \left(\begin{array}{c|c} A & 0\\ \hline 0 & B\end{array}\right).\end{array}$$
Thus, 
$$\Inf_{\UT_k(q)\times\UT_{n-k}(q)}^{\UT_n(q)}(\kappa_\mu\times\kappa_\nu)(u_\lambda)=\left\{\begin{array}{ll} 1, & \begin{array}{@{}l}\text{if $\mu=\{i\larc{a}j\in \lambda\mid i,j\in [k]\}$,}\\ \text{and $\nu+k=\{i\larc{a}j\in \lambda\mid i,j\in [k]^c\}$,}\end{array}\\
0, & \text{otherwise,}\end{array}\right.
$$
as desired.

(b)  Let $\mu\in \cS_A(q)$, and $\nu\in \cS_{A^c}(q)$, and let $u_{\mu\times \nu}$ be the corresponding superclass representative for $\UT_{A|A^c}(q)$.  Note that 
\begin{align*}
\Res_{\UT_{A|A^c}(q)}^{\UT_n(q)}(\kappa_\lambda)(u_{\mu\times \nu})&=\left\{\begin{array}{ll} 1, & \text{if $\kappa_\lambda(u_{\mu\times \nu})=1$,}\\ 0, & \text{otherwise,}\end{array}\right.\\
&=\left\{\begin{array}{ll} 1, & \begin{array}{@{}l} \text{if $\mu=\{i\larc{a}j\in \lambda\mid i,j\in A\}$ and}\\ \nu=\{i\larc{a}j\in \lambda\mid i,j\in A^c\},\end{array}\\ 0, & \text{otherwise.}\end{array}\right.
\end{align*}
Thus, if $\mu=\{i\larc{a}j\in \lambda\mid i,j\in A\}$ and $\nu=\{i\larc{a}j\in \lambda\mid i,j\in A^c\}$, then
$$\Res_{\UT_{A|A^c}(q)}^{\UT_n(q)}(\kappa_\lambda)=\kappa_{\mu}\otimes \kappa_{\nu},$$
and the result follows by applying the $\st_J$ map.
\end{proof}

\begin{example}  We have
\begin{align*}
\kappa\hspace{-.1cm}
\begin{tikzpicture}[baseline=.2cm]
	\foreach \x in {1,2,3} 
		\node (\x) at (\x/4,0) [inner sep=0pt] {$\scs\bullet$};
	\foreach \x in {1,2,3} 
		\node at (\x/4,-.2) {$\scscs\x$};
	\draw (1) .. controls (1.75/4,.3) and (2.25/4,.3) ..  node [above=-2pt] {$\scs a$} (3); 
\end{tikzpicture}
\cdot
\kappa\hspace{-.1cm}
\begin{tikzpicture}[baseline=.2cm]
	\foreach \x in {1,2,3,4} 
		\node (\x) at (\x/4,0) [inner sep=0pt] {$\scs\bullet$};
	\foreach \x in {1,2,3,4} 
		\node at (\x/4,-.2) {$\scscs\x$};
	\draw (1) .. controls (1.25/4,.25) and (1.75/4,.25) ..  node [above=-2pt] {$\scs b$} (2); 
	\draw (2) .. controls (2.75/4,.3) and (3.25/4,.3) ..  node [above=-2pt] {$\scs c$} (4);
\end{tikzpicture}
=&
\kappa \hspace{-.1cm}
\begin{tikzpicture}[baseline=.2cm]
	\foreach \x in {1,...,7} 
		\node (\x) at (\x/4,0) [inner sep=0pt] {$\scs\bullet$};
	\foreach \x in {1,...,7} 
		\node at (\x/4,-.2) {$\scscs\x$};
	\draw (1) .. controls (1.75/4,.3) and (2.25/4,.3) ..  node [above=-2pt] {$\scs a$} (3); 
	\draw (4) .. controls (4.25/4,.25) and (4.75/4,.25) ..  node [above=-2pt] {$\scs b$} (5); 
	\draw (5) .. controls (5.75/4,.3) and (6.25/4,.3) ..  node [above=-2pt] {$\scs c$} (7);
\end{tikzpicture}+\sum_{d\in \FF_q^\times}
\bigg(
\kappa \hspace{-.1cm}
\begin{tikzpicture}[baseline=.2cm]
	\foreach \x in {1,...,7} 
		\node (\x) at (\x/4,0) [inner sep=0pt] {$\scs\bullet$};
	\foreach \x in {1,...,7} 
		\node at (\x/4,-.2) {$\scscs\x$};
	\draw (1) .. controls (1.75/4,.3) and (2.25/4,.3) ..  node [above=-2pt] {$\scs a$} (3); 
	\draw (3) .. controls (3.25/4,.25) and (3.75/4,.25) ..  node [above=-2pt] {$\scs d$} (4); 
	\draw (4) .. controls (4.25/4,.25) and (4.75/4,.25) ..  node [above=-2pt] {$\scs b$} (5); 
	\draw (5) .. controls (5.75/4,.3) and (6.25/4,.3) ..  node [above=-2pt] {$\scs c$} (7);
\end{tikzpicture}+
\kappa\hspace{-.1cm}
\begin{tikzpicture}[baseline=.2cm]
	\foreach \x in {1,...,7} 
		\node (\x) at (\x/4,0) [inner sep=0pt] {$\scs\bullet$};
	\foreach \x in {1,...,7} 
		\node at (\x/4,-.2) {$\scscs\x$};
	\draw (1) .. controls (1.75/4,.3) and (2.25/4,.3) ..  node [above=-2pt] {$\scs a$} (3); 
	\draw (2) .. controls (2.75/4,.3) and (3.25/4,.3) ..  node [above=-2pt] {$\scs d$} (4); 
	\draw (4) .. controls (4.25/4,.25) and (4.75/4,.25) ..  node [above=-2pt] {$\scs b$} (5); 
	\draw (5) .. controls (5.75/4,.3) and (6.25/4,.3) ..  node [above=-2pt] {$\scs c$} (7);
\end{tikzpicture}
+\kappa\hspace{-.1cm}
\begin{tikzpicture}[baseline=.2cm]
	\foreach \x in {1,...,7} 
		\node (\x) at (\x/4,0) [inner sep=0pt] {$\scs\bullet$};
	\foreach \x in {1,...,7} 
		\node at (\x/4,-.2) {$\scscs\x$};
	\draw (1) .. controls (1.75/4,.3) and (2.25/4,.3) ..  node [above=-2pt] {$\scs a$} (3); 
	\draw (3) .. controls (4/4,.65) and (5/4,.65) ..  node [above=-2pt] {$\scs d$} (6); 
	\draw (4) .. controls (4.25/4,.25) and (4.75/4,.25) ..  node [above=-2pt] {$\scs b$} (5); 
	\draw (5) .. controls (5.75/4,.3) and (6.25/4,.3) ..  node [above=-2pt] {$\scs c$} (7);
\end{tikzpicture}
+\kappa\hspace{-.1cm}
\begin{tikzpicture}[baseline=.2cm]
	\foreach \x in {1,...,7} 
		\node (\x) at (\x/4,0) [inner sep=0pt] {$\scs\bullet$};
	\foreach \x in {1,...,7} 
		\node at (\x/4,-.2) {$\scscs\x$};
	\draw (1) .. controls (1.75/4,.3) and (2.25/4,.3) ..  node [above=-2pt] {$\scs a$} (3); 
	\draw (2) .. controls (3.25/4,.7) and (4.75/4,.7) ..  node [above=-2pt] {$\scs d$} (6); 
	\draw (4) .. controls (4.25/4,.25) and (4.75/4,.25) ..  node [above=-2pt] {$\scs b$} (5); 
	\draw (5) .. controls (5.75/4,.3) and (6.25/4,.3) ..  node [above=-2pt] {$\scs c$} (7);
\end{tikzpicture}\bigg)\\
&+\sum_{d,e\in\FF_q^\times}\bigg( 
\kappa\hspace{-.1cm}
\begin{tikzpicture}[baseline=.2cm]
	\foreach \x in {1,...,7} 
		\node (\x) at (\x/4,0) [inner sep=0pt] {$\scs\bullet$};
	\foreach \x in {1,...,7} 
		\node at (\x/4,-.2) {$\scscs\x$};
	\draw (1) .. controls (1.75/4,.3) and (2.25/4,.3) ..  node [above=-2pt] {$\scs a$} (3); 
	\draw (2) .. controls (3.25/4,.7) and (4.75/4,.7) ..  node [above=-2pt] {$\scs d$} (6); 
	\draw (3) .. controls (3.25/4,.25) and (3.75/4,.25) ..  node [above=-2pt] {$\scs e$} (4); 
	\draw (4) .. controls (4.25/4,.25) and (4.75/4,.25) ..  node [above=-2pt] {$\scs b$} (5); 
	\draw (5) .. controls (5.75/4,.3) and (6.25/4,.3) ..  node [above=-2pt] {$\scs c$} (7);
\end{tikzpicture}+\kappa\hspace{-.1cm}
\begin{tikzpicture}[baseline=.2cm]
	\foreach \x in {1,...,7} 
		\node (\x) at (\x/4,0) [inner sep=0pt] {$\scs\bullet$};
	\foreach \x in {1,...,7} 
		\node at (\x/4,-.2) {$\scscs\x$};
	\draw (1) .. controls (1.75/4,.3) and (2.25/4,.3) ..  node [above=-2pt] {$\scs a$} (3); 
	\draw (2) .. controls (2.75/4,.3) and (3.25/4,.3) ..  node [above=-2pt] {$\scs d$} (4); 
	\draw (3) .. controls (4/4,.65) and (5/4,.65) ..  node [above=-2pt] {$\scs e$} (6); 
	\draw (4) .. controls (4.25/4,.25) and (4.75/4,.25) ..  node [above=-2pt] {$\scs b$} (5); 
	\draw (5) .. controls (5.75/4,.3) and (6.25/4,.3) ..  node [above=-2pt] {$\scs c$} (7);
\end{tikzpicture}\bigg).
\end{align*}
and 
\begin{align*}
\Delta\bigg(
\kappa\hspace{-.1cm}
\begin{tikzpicture}[baseline=.2cm]
	\foreach \x in {1,2,3,4} 
		\node (\x) at (\x/4,0) [inner sep=0pt] {$\scs\bullet$};
	\foreach \x in {1,2,3,4} 
		\node at (\x/4,-.2) {$\scscs\x$};
	\draw (1) .. controls (1.75/4,.25) and (3.25/4,.25) ..  node [above=-2pt] {$\scs a$} (4); 
\end{tikzpicture}\bigg)=&
\kappa\hspace{-.1cm}
\begin{tikzpicture}[baseline=.2cm]
	\foreach \x in {1,2,3,4} 
		\node (\x) at (\x/4,0) [inner sep=0pt] {$\scs\bullet$};
	\foreach \x in {1,2,3,4} 
		\node at (\x/4,-.2) {$\scscs\x$};
	\draw (1) .. controls (1.75/4,.25) and (3.25/4,.25) ..  node [above=-2pt] {$\scs a$} (4); 
\end{tikzpicture}\otimes \kappa_\emptyset
+
2\kappa\hspace{-.1cm}
\begin{tikzpicture}[baseline=.2cm]
	\foreach \x in {1,2,3} 
		\node (\x) at (\x/4,0) [inner sep=0pt] {$\scs\bullet$};
	\foreach \x in {1,2,3} 
		\node at (\x/4,-.2) {$\scscs\x$};
	\draw (1) .. controls (1.75/4,.25) and (2.25/4,.25) ..  node [above=-2pt] {$\scs a$} (3); 
\end{tikzpicture}
\otimes \kappa\hspace{-.1cm}
\begin{tikzpicture}[baseline=.1cm]
	\foreach \x in {1} 
		\node (\x) at (\x/4,0) [inner sep=0pt] {$\scs\bullet$};
	\foreach \x in {1} 
		\node at (\x/4,-.2) {$\scscs\x$};
\end{tikzpicture}+\kappa\hspace{-.1cm}
\begin{tikzpicture}[baseline=.2cm]
	\foreach \x in {1,2} 
		\node (\x) at (\x/4,0) [inner sep=0pt] {$\scs\bullet$};
	\foreach \x in {1,2} 
		\node at (\x/4,-.2) {$\scscs\x$};
	\draw (1) .. controls (1.25/4,.25) and (1.75/4,.25) ..  node [above=-2pt] {$\scs a$} (2); 
\end{tikzpicture}
\otimes \kappa\hspace{-.1cm}
\begin{tikzpicture}[baseline=.1cm]
	\foreach \x in {1,2} 
		\node (\x) at (\x/4,0) [inner sep=0pt] {$\scs\bullet$};
	\foreach \x in {1,2} 
		\node at (\x/4,-.2) {$\scscs\x$};
\end{tikzpicture}\\
&+  \kappa\hspace{-.1cm}
\begin{tikzpicture}[baseline=.1cm]
	\foreach \x in {1,2} 
		\node (\x) at (\x/4,0) [inner sep=0pt] {$\scs\bullet$};
	\foreach \x in {1,2} 
		\node at (\x/4,-.2) {$\scscs\x$};
\end{tikzpicture}\otimes \kappa\hspace{-.1cm}
\begin{tikzpicture}[baseline=.2cm]
	\foreach \x in {1,2} 
		\node (\x) at (\x/4,0) [inner sep=0pt] {$\scs\bullet$};
	\foreach \x in {1,2} 
		\node at (\x/4,-.2) {$\scscs\x$};
	\draw (1) .. controls (1.25/4,.25) and (1.75/4,.25) ..  node [above=-2pt] {$\scs a$} (2); 
\end{tikzpicture}+2 \kappa\hspace{-.1cm}
\begin{tikzpicture}[baseline=.1cm]
	\foreach \x in {1} 
		\node (\x) at (\x/4,0) [inner sep=0pt] {$\scs\bullet$};
	\foreach \x in {1} 
		\node at (\x/4,-.2) {$\scscs\x$};
\end{tikzpicture}\otimes\kappa\hspace{-.1cm}
\begin{tikzpicture}[baseline=.2cm]
	\foreach \x in {1,2,3} 
		\node (\x) at (\x/4,0) [inner sep=0pt] {$\scs\bullet$};
	\foreach \x in {1,2,3} 
		\node at (\x/4,-.2) {$\scscs\x$};
	\draw (1) .. controls (1.75/4,.25) and (2.25/4,.25) ..  node [above=-2pt] {$\scs a$} (3); 
\end{tikzpicture}+\kappa_\emptyset \otimes \kappa\hspace{-.1cm}
\begin{tikzpicture}[baseline=.2cm]
	\foreach \x in {1,2,3,4} 
		\node (\x) at (\x/4,0) [inner sep=0pt] {$\scs\bullet$};
	\foreach \x in {1,2,3,4} 
		\node at (\x/4,-.2) {$\scscs\x$};
	\draw (1) .. controls (1.75/4,.25) and (3.25/4,.25) ..  node [above=-2pt] {$\scs a$} (4); 
\end{tikzpicture}.
\end{align*}

\end{example}

By comparing Proposition \ref{CombinatorialSuperclassFunctions} to (\ref{mult_monomial}) and (\ref{MonomialCoproduct}), we
obtain the following theorem.

\begin{theorem}\label{q=2Correspondence}
For $q=2$, the map
$$\begin{array}{rccc} \ch: & \SC & \longrightarrow & \NCSym\\ 
& \kappa_\mu & \mapsto & \bm_\mu\end{array}$$
is a Hopf algebra isomorphism.
\end{theorem}

Note that although we did not assume for  the theorem that $\SC$ is a Hopf algebra, the fact that  $\ch$ preserves the Hopf operations implies that $\SC$ for $q=2$ is indeed a Hopf algebra. The general result will follow from Section~\ref{ColoredCorrespondence}. 

\begin{corollary}\label{PiIsHopfAlgebra}
The algebra $\SC$ with product given by (\ref{InflationProduct}) and coproduct given by (\ref{RestrictionCoProduct}) is a Hopf algebra.
\end{corollary}

\begin{remarks}\hfill

\begin{enumerate}
\item[(a)] Note that the isomorphism of Theorem \ref{q=2Correspondence} is not in any way canonical.    In fact, the automorphism group of $\NCSym$ is rather large, so there are many possible isomorphisms.  For our chosen isomorphism, there is no nice interpretation for the image of the supercharacters under the  isomorphism of Theorem \ref{q=2Correspondence}.   Even less pleasant, when one composes it with the map
$$\NCSym \longrightarrow \Sym$$
that allows variables to commute (see \cite{Do,SR}), one in fact obtains that the supercharacters are not Schur positive. But, exploration with Sage suggests that it may be possible to choose an isomorphism such that the image of the supercharacters are Schur positive. 
\item[(b)] Although the antipode is determined by the bialgebra structure of $\NCSym$, explicit expressions are not well understood.   However, there are a number of forthcoming papers (e.g. \cite{ABT, LM}) addressing this situation.
\item[(c)]  In~\cite{ABS06}, the authors considered the category
of combinatorial Hopf algebras consisting of pairs $(H,\zeta)$,
where $H$ is a graded connected Hopf algebra and $\zeta\colon H\to
\CC$ is a character (an algebra homomorphism). As remarked in~\cite{BLL}, every graded Hopf algebra arising from representation theory  yields a canonical character. This is still true for $\SC$.  For all $n\ge 0$ consider the dual to the trivial supercharacter $(\chi^{\emptyset_n})^*\in \SC_n^*$. It follows from Section \ref{DualSCSection} below that
$$\Delta\big((\chi^{\emptyset_n})^*\big) = \sum_{k =0}^n(\chi^{\emptyset_k})^* \otimes
       (\chi^{\emptyset_{n-k}})^*,$$
which implies  that the map $\zeta:\SC\rightarrow \CC$ given by
$$\zeta(\chi)=\langle (\chi^{\emptyset_n})^*,\chi\rangle,\qquad \text{where $\chi\in \SC_n$,}$$
is a character. We thus have that
$(\SC,\zeta)$ is a combinatorial Hopf algebra in the sense of
\cite{ABS06}.  This connection awaits further exploration.
\end{enumerate}
\end{remarks}

The Hopf algebra $\SC$ has a number of natural Hopf subalgebras.  One of particular interest is the subspace spanned by linear characters (characters with degree 1).   In fact, for this supercharacter theory every linear character of $U_n$ is a supercharacter and by (\ref{SupercharacterDegree}) these are exactly indexed by the set
$$\mathcal{L}_n=\{\lambda\in \cS_n(q)\mid i\larc{a}j\in \lambda\text{ implies } j=i+1\}.$$

\begin{corollary}
For $q=2$, the Hopf subalgebra 
$${\mathrm{\mathbf{LSC}}}=\CC\spanning\{\chi^\lambda \mid  i\larc{1}j\in \lambda\text{ implies } j=i+1\},$$
is isomorphic to the Hopf algebra of noncommutative symmetric functions $\mathbf{Sym}$ studied in \cite{NCSF1}.
\end{corollary}

\begin{proof} Let the \emph{length} of an arc $i\larc{a}j$ be $j-i$.  By inspection of the product and coproduct of $\SC$, we observe that an arc $i\larc{a}j$ never increases in length. Since ${\mathrm{\mathbf{LSC}}}$
is the linear span of supercharacters indexed by set partitions with arcs of length at most $1$, it is clearly a Hopf subalgebra.

By (\ref{SupercharacterFormula}), for $[n]=\{i\larc{1}(i+1) \mid 1\le i<n\}$, we have $\langle \chi^\lambda,\kappa_{[n]}\rangle=0$ unless $\lambda\in \cL_n$.   Thus, the superclass functions $\kappa_{[n]}\in{\mathrm{\mathbf{LSC}}}$.  Furthermore, if we order by refinement in $\SC_n$, then the set of products 
$$\{\kappa_{[k_1]}\kappa_{[k_2]}\cdots \kappa_{[k_\ell]}\mid k_1+k_2+\cdots+k_\ell=n, \ell\geq 1\}$$
have an upper-triangular decomoposition in terms of the $\kappa_\mu$.  Therefore, the elements $\kappa_{[n]}$ are algebraically independent in ${\mathrm{\mathbf{LSC}}}$, and $\mathrm{\mathbf{LSC}}$ contains the free algebra
$$ \CC\langle \kappa_{[1]},\kappa_{[2]},\ldots\rangle.$$
Note that every element  $\lambda\in \cL_n$ is of the form 
$$\lambda=[1\larc{\larc{}}k_1]\cup \big[ (k_1+1)\larc{\larc{}} (k_1+k_2)\big]\cup\ldots\cup \big[ (n-k_\ell)\larc{\larc{}} n\big],$$ 
where $[i\larc{\larc{}} j]= \{i\larc{1}(i+1),(i+1)\larc{1}(i+2),\ldots,(j-1)\larc{1}j\}.$
Thus,
$$|\cL_n|=\dim\big(\CC\spanning\{\kappa_{[k_1]}\kappa_{[k_2]}\cdots \kappa_{[k_\ell]}\mid k_1+k_2+\cdots+k_\ell=n, \ell\geq 1\}\big),$$ 
implies
$${\mathrm{\mathbf{LSC}}}=\CC\langle \kappa_{[1]},\kappa_{[2]},\ldots\rangle.$$
  
On the other hand,  \cite{NCSF1} describes $\mathbf{Sym}$  as follows:
  $$\mathbf{Sym} = \CC\langle \Psi_1,\Psi_2,\ldots\rangle$$
is the free (non-commutative) algebra with $\deg(\Psi_k)=k$ and coproduct given by
  $$\Delta(\Psi_k) = 1\otimes \Psi_k + \Psi_k\otimes 1.$$
Hence, the map $\kappa_{[k]}\mapsto\Psi_k$ gives the desired isomorphism.
\end{proof}

\begin{remark}
  In fact, for each $k\in \ZZ_{\geq 0}$ the space
$$\SC^{(k)}=\CC\spanning\{\chi^\lambda\mid i\larc{}j\in \lambda\text{ implies } j-i\leq k\}$$
is a Hopf subalgebra. This gives an unexplored filtration of Hopf algebras which interpolate between $\mathrm{\mathbf{LSC}}$ and $\SC$.

\end{remark}

\subsection{A colored version of $\WSym$} \label{ColoredCorrespondence}

There are several natural ways to color a combinatorial Hopf algebra; for example see~\cite{BH}. The Hopf algebra $\SC$ for general $q$ is a Hopf subalgebra of the ``naive" coloring of $\NCSym$.

Let $C_r=\langle\zeta\rangle$ be a cyclic group of order $r$ (which in our case will eventually be $r=q-1$).  We expand our set of variables $A=\{a_1,a_2,\ldots\}$  by letting
$$A^{(r)}=A\times C_r.$$
We view the elements of $C_r$ as colors that decorate the
variables of $A$.  The group $\SG_A$ acts on the first coordinate of the set $A^{(r)}$. That is, $\sigma(a_i,\zeta^j)=(\sigma(a_i),\zeta^j)$. With this action, we define 
$\tilde\NCSym^{(r)}$ as the set of bounded formal power series in $A^r$ invariant under the action of $\SG_A$. As before, we assume that $A$ is infinite and the space $\tilde\NCSym^{(r)}$ is a graded algebra based on $r$-colored set partitions $(\mu,(\zeta_1,\ldots, \zeta_n))$ where $\mu$ is a set partition of the set $\{1,2,\ldots,n\}$ and  $(\zeta_1,\ldots, \zeta_n)\in C_r^n$. It has a basis of monomial elements given by
\begin{equation}
\bm_{\mu,(\zeta_1,\ldots, \zeta_n)}\left(A^{(r)}\right)=\sum_{w\in O_{\mu,(\zeta_1,\ldots, \zeta_n)}}w,
\end{equation}
where $O_{\mu,(\zeta_1,\ldots, \zeta_n)}$ is the orbit of $\SG_A$ indexed by $(\mu,(\zeta_1,\ldots, \zeta_n))$. More precisely, it is the set of words 
$w=(a_{i_1},\zeta_1)(a_{i_2},\zeta_2)\ldots (a_{i_n},\zeta_n)$ on the alphabet $A^{(r)}$ such that $a_i=a_j$ if and only if $i$ and $j$ are
in the same block of $\mu$. The concatenation product on $\KK\langle A^{(r)}\rangle$ gives us the following combinatorial description of the product  in $\tilde\NCSym^{(r)}$ in the monomial basis. If $\lambda\vdash [k]$ and $\mu\vdash[n-k]$, then
\begin{equation}
\bm_{\lambda,(\zeta_1,\ldots, \zeta_k)} \bm_{\mu,(\zeta'_1,\ldots, \zeta'_{n-k})}=\sum_{\nu\vdash [n]\atop \nu\wedge ([k]|[n-k])=\lambda|\mu}
     \bm_{\nu,(\zeta_1,\ldots, \zeta_k,\zeta'_1,\ldots, \zeta'_{n-k})}\,.
\end{equation}
This is just a colored version of (\ref{mult_monomial}).

As before, we define a coproduct by
\begin{equation}
\Delta(f) =  \sum_k f_k'\otimes f_k''    \qquad \text{where}\qquad           f\left(A'^{(r)}+A''^{(r)}\right)   = \sum_k f_k'\left(A'^{(r)}\right)f_k''\left(A''^{(r)}\right)
\end{equation}
and $A'^{(r)}+A''^{(r)}$ denotes the disjoint union of two mutually commuting copies of $A^{(r)}$. This is clearly
coassociative and a morphism of algebras;  hence, $\NCSym^r$ is a graded Hopf algebra.
The coproduct of a monomial function is 
\begin{equation}
\Delta \bm_{\lambda,(\zeta_1,\ldots, \zeta_{n})} = \sum_{\mu\vee\nu=\lambda}\bm_{\st(\mu),\zeta\big|_{\mu}}\otimes \bm_{\st(\nu),\zeta\big|_{\nu}},
\end{equation}
where $\zeta\big|_{\mu}$ denotes the subsequence $(\zeta_{i_1},\zeta_{i_2},...)$ with $i_1<i_2<\cdots$ and $i_j$ appearing in a block of $\mu$.
The complement sequence is $\zeta\big|_{\nu}$.
This coproduct is  cocommutative. With the unit $u\colon 1\mapsto 1$ and the counit $\epsilon\colon f(A^{(r)})\mapsto f(0,0,\ldots)$ we have that $\tilde\NCSym^{(r)}$ is a connected graded bialgebra and therefore a graded Hopf algebra.

Now we describe a Hopf subalgebra of this space indexed by $ \cS_n(q)$ for $n\ge 0$.
For $(D,\phi)\in \cS_n(q)$, let
$$\bk_{(D,\phi)}= \sum_{(\zeta_1,\ldots,\zeta_n)\in C_r^{n}\atop \zeta_{j}/\zeta_i=\phi(i,j)} \bm_{\pi(D,\phi),(\zeta_1,\ldots,\zeta_n)}\,,$$
where $\pi(D,\phi)$ is the underlying set partition of $D$ (as in (\ref{UnderlyingSetPartition})).

\begin{proposition}\label{prop:NCSym_q}
The space
$$\NCSym^{(q-1)}=\CC\spanning\{\bk_{(D,\phi)}\mid (D,\phi)\in \cM_n(q),n\in\ZZ_{\geq 0}\}$$
is a Hopf subalgebra of $\tilde\NCSym^{(q-1)}$.   For  $\mu\in \cS_k(q)$, $\nu\in \cS_{n-k}(q)$ the product is given by
\begin{equation}\label{kmult}
\bk_\mu\cdot \bk_\nu=\sum_{\lambda=\mu\sqcup \gamma\sqcup (k+\nu)\in \cS_n(q)\atop   i\slarc{\scscs a}l\in \gamma \text{ implies } i\leq k < l} \bk_\lambda,
\end{equation}
 and for $\lambda\in \cS_n(q)$, the coproduct is given by
\begin{equation}\label{kcomult}
\Delta(\bk_\lambda)=\sum_{{\lambda=\mu\sqcup \nu\atop \mu\in \cS_A(q), \nu\in \cS_{A^c}(q)}\atop A\subseteq \{1,2,\ldots,n\}} \bk_{\mathrm{st}_A(\mu)}\otimes \bk_{\mathrm{st}_{A^c}(\nu)}.
\end{equation}
\end{proposition}

\begin{proof} It is sufficient to show that $\NCSym^{(q-1)}$ is closed under product and coproduct. Thus, it is enough to show that (\ref{kmult}) and (\ref{kcomult}) are valid. For  $\mu=(D,\phi)\in \cS_k(q)$, $\nu=(D',\phi')\in \cS_{n-k}(q)$ let $\pi(\mu)$ and $\pi(\nu)$ be the underlying set partitions of $\mu$ and $\nu$, respectively. We have 
\begin{align*}
\bk_\mu\cdot \bk_\nu&= \bigg( \sum_{(\zeta_1,\ldots,\zeta_k)\in C_r^{k}\atop \zeta_{j}/\zeta_i=\phi(i,j)} \bm_{\pi(\mu),(\zeta_1,\ldots,\zeta_k)} \bigg)
                                           \bigg( \sum_{(\zeta_{k+1},\ldots,\zeta_n)\in C_r^{n-k}\atop \zeta_{k+j}/\zeta_{k+i}=\phi'(i,j)} \bm_{\pi(\nu),(\zeta_{k+1},\ldots,\zeta_n)} \bigg)\\
                                     &=  \sum_{(\zeta_1,\ldots,\zeta_n)\in C_r^{n}\atop { \zeta_{j}/\zeta_i=\phi(i,j),\, i<j\le k \atop  \zeta_{j}/\zeta_{i}=\phi'(i-k,j-k), k<i<j}} 
                                            \sum_{\rho\vdash [n]\atop \rho\wedge ([k]|[n-k])=\pi(\mu)|\pi(\nu)}    \bm_{\rho,(\zeta_1,\ldots,  \zeta_{n})}\\ 
                                     &= \sum_{\lambda=\mu\sqcup \gamma\sqcup (k+\nu)\in \cS_n(q)\atop   i\slarc{\scscs a}l\in \gamma \text{ implies } i\leq k < l} \bk_\lambda.
\end{align*}
In the second equality, the second sum ranges over set partitions $\rho$ obtained by grouping some block of $\pi(\mu)$ with some block of $\pi(\nu)$.
These set partitions can be thought of as collections of arcs $i\larc{}j$ with $1\le i\le k<j\le n$.
In the last equality, we group together the terms $\bm_{\rho,(\zeta_1,\ldots,  \zeta_{n})}$ such that $ \zeta_{j}/\zeta_i=\phi''(i,j)$ for $i\le k<j$.

Now for $\lambda=(D,\phi)\in \cS_n(q)$,
\begin{align*}
\Delta(\bk_\lambda)  &=\Delta \bigg( \sum_{(\zeta_1,\ldots,\zeta_n)\in C_r^{n}\atop \zeta_{j}/\zeta_i=\phi(i,j)} \bm_{\pi(\lambda),(\zeta_1,\ldots,\zeta_n)} \bigg)\\
                                     &=  \sum_{(\zeta_1,\ldots,\zeta_n)\in C_r^{n}\atop \zeta_{j}/\zeta_i=\phi(i,j)} 
                                           \sum_{\mu\vee\nu=\lambda}\bm_{\st(\mu),\zeta\big|_{\mu}}\otimes \bm_{\st(\nu),\zeta\big|_{\nu}}\\ 
                                     &=\sum_{{\lambda=\mu\sqcup \nu\atop \mu\in \cS_A(q), \nu\in \cS_{A^c}(q)}\atop A\subseteq \{1,2,\ldots,n\}}
                                             \bk_{\mathrm{st}_A(\mu)}\otimes \bk_{\mathrm{st}_{A^c}(\nu)}.
\end{align*}
\end{proof}

Comparing Proposition~\ref{prop:NCSym_q} and Proposition~\ref{SuperclassFunctionsOperations}, we obtain

\begin{theorem}
The map
$$\begin{array}{rccc} \ch: & \SC & \longrightarrow & \WSym^{(q-1)}\\ &\kappa_\mu & \mapsto & \bk_{(D_\mu,\phi_\mu)}\end{array}$$
is an isomorphism of Hopf algebras. In particular, $\SC$ is a Hopf algebra for any $q$.
\end{theorem}

\begin{remark}
  As in the $q=2$ case, for each $k\in \ZZ_{\geq 0}$ the space
$$\SC^{(k)}=\CC\spanning\{\chi^\lambda\mid i\larc{a}j\in \lambda\text{ implies } j-i\leq k\}$$
is a Hopf subalgebra of $\SC$.  For $k=1$, this gives a $q$-version of the Hopf algebra of noncommutative symmetric functions.
\end{remark}

\section{The dual Hopf algebras $\SC^*$ and $\NCSym^*$}

This section explores the dual Hopf algebras $\SC^*$ and $\NCSym^*$. We begin by providing representation theoretic interpretations of the product and coproduct of $\SC^*$, followed by a concrete realization of $\SC^*$ and $\NCSym^*$.

\subsection{The Hopf algebra $\SC^*$}\label{DualSCSection}

As a graded vector space, $\SC^*$ is the vector space dual to $\SC$:
\begin{align*}
\SC^* &= \bigoplus_{n\geq 0} \SC_n^*\\ 
&=\CC\spanning\{\kappa_\mu^* \mid \mu \in \cS_n(q),n\in \ZZ_{\geq 0}\}\\
&=\CC\spanning\{(\chi^\lambda)^* \mid \lambda \in \cS_n(q),n\in \ZZ_{\geq 0}\}.
\end{align*}
We may use the inner product (\ref{InnerProductSupercharacters}) to identify $\SC^*$ with $\SC$ as graded vector spaces.
Under this identification, the basis element dual to $\kappa_\mu$ with respect to the inner product (\ref{InnerProductSupercharacters}) is
\begin{align*}
\kappa_\mu^* &= z_\mu\kappa_\mu, && \text{where } z_\mu=\frac{|\UT_{|\mu|}(q)|}{|\UT_{|\mu|}(q)(u_\mu-1)\UT_{|\mu|}(q)|},
\end{align*}
and the basis element dual to $\chi^\lambda$ is
\begin{align*}
(\chi^\lambda)^*&= q^{-C(\lambda)}\chi^\lambda, && \text{where } C(\lambda)=\#\{(i,j,k,l)\mid i\larc{a}k,j\larc{b}l\in \lambda\}.
\end{align*}
The product on $\SC^*$ is given by
$$\chi\cdot \psi = \sum_{{J=A|A^c\atop A\subseteq [a+b]}\atop |A|=a} {}^J\SInd_{\UT_{a}(q)\times \UT_{b}(q)}^{\UT_{a+b}(q)}(\chi \times \psi),\qquad \text{for $\chi\in \SC_a^*,\psi\in \SC_{b}^*$},$$
and the coproduct by
$$\Delta(\chi)=\sum_{k=0}^n \Def_{\UT_{(k,n-k)}(q)}^{\UT_n(q)}(\chi), \qquad \text{for $\chi\in\SC_n^*$}.$$  

\begin{proposition} 
The product and coproduct of $\SC^*$ in the $\kappa^*$ basis is given by
\begin{enumerate}
\item[(a)] for $\mu\in \cS_k(q)$ and $\nu\in\cS_{n-k}(q)$,
$$\kappa_\mu^*\cdot \kappa_\nu^*=\sum_{J\subseteq [n]\atop |J|=k} \kappa_{\mathrm{st}^{-1}_J(\mu)\cup\mathrm{st}^{-1}_{J^c}(\nu)}^*,$$
\item[(b)] for $\lambda\in \cS_n(q)$,
$$\Delta(\kappa_\lambda^*)=\sum_{k=0}^n \kappa_{\lambda_{[k]}}^*\otimes \kappa_{\lambda_{[k]^c}}^*,\qquad \text{where}\qquad \lambda_J=\{i\larc{a}j\in \lambda\mid i,j\in J\}.$$
\end{enumerate}
\end{proposition}
\begin{proof}
This result follows from Proposition \ref{SuperclassFunctionsOperations}, and the duality results
\begin{itemize}
\item $\dd\left\langle {}^J\SInd^{\UT_n(q)}_{\st_J(\UT_J(q))}(\psi), \chi\right\rangle=\left\langle \psi, {}^J\Res^{\UT_n(q)}_{\st_J(\UT_J(q))}(\chi)\right\rangle$,
\item $\dd \left\langle \Inf^{\UT_n(q)}_{\UT_{(m_1,\ldots,m_\ell)}(q)}(\psi), \chi\right\rangle=\left\langle \psi, \Def^{\UT_n(q)}_{\UT_{(m_1,\ldots,m_\ell)}(q)}(\chi)\right\rangle$,
\item $\left\langle \kappa_\mu,\kappa_\nu^*\right\rangle=\delta_{\mu\nu}$.\qedhere
\end{itemize}
\end{proof}
Note that the Hopf algebra $\SC^*$ is commutative, but not cocommutative.

\begin{example}  We have
\begin{align*}
\kappa^*\hspace{-.2cm}
\begin{tikzpicture}[baseline=.2cm]
	\foreach \x in {1,2} 
		\node (\x) at (\x/4,0) [inner sep=0pt] {$\scs\bullet$};
	\foreach \x in {1,2} 
		\node at (\x/4,-.2) {$\scscs\x$};
	\draw (1) .. controls (1.25/4,.25) and (1.75/4,.25) ..  node [above=-2pt] {$\scs a$} (2); 
\end{tikzpicture}
\cdot
\kappa^*\hspace{-.2cm}
\begin{tikzpicture}[baseline=.2cm]
	\foreach \x in {1,2,3} 
		\node (\x) at (\x/4,0) [inner sep=0pt] {$\scs\bullet$};
	\foreach \x in {1,2,3} 
		\node at (\x/4,-.2) {$\scscs\x$};
	\draw (1) .. controls (1.75/4,.3) and (2.25/4,.3) ..  node [above=-2pt] {$\scs b$} (3);
\end{tikzpicture}
=&
\kappa^*\hspace{-.2cm}
\begin{tikzpicture}[baseline=.2cm]
	\foreach \x in {1,...,5} 
		\node (\x) at (\x/4,0) [inner sep=0pt] {$\scs\bullet$};
	\foreach \x in {1,...,5} 
		\node at (\x/4,-.2) {$\scscs\x$};
	\draw (1) .. controls (1.25/4,.25) and (1.75/4,.25) ..  node [above=-2pt] {$\scs a$} (2); 
	\draw (3) .. controls (3.75/4,.3) and (4.25/4,.3) ..  node [above=-2pt] {$\scs b$} (5);
\end{tikzpicture}
+\kappa^*\hspace{-.2cm}
\begin{tikzpicture}[baseline=.2cm]
	\foreach \x in {1,...,5} 
		\node (\x) at (\x/4,0) [inner sep=0pt] {$\scs\bullet$};
	\foreach \x in {1,...,5} 
		\node at (\x/4,-.2) {$\scscs\x$};
	\draw (1) .. controls (1.75/4,.3) and (2.25/4,.3) ..  node [above=-2pt] {$\scs a$} (3); 
	\draw (2) .. controls (3/4,.35) and (4/4,.35) ..  node [above=-2pt] {$\scs b$} (5);
\end{tikzpicture}
+\kappa^*\hspace{-.2cm}
\begin{tikzpicture}[baseline=.2cm]
	\foreach \x in {1,...,5} 
		\node (\x) at (\x/4,0) [inner sep=0pt] {$\scs\bullet$};
	\foreach \x in {1,...,5} 
		\node at (\x/4,-.2) {$\scscs\x$};
	\draw (1) .. controls (2/4,.35) and (3/4,.35) ..  node [above=-2pt] {$\scs a$} (4); 
	\draw (2) .. controls (3/4,.35) and (4/4,.35) ..  node [above=-2pt] {$\scs b$} (5);
\end{tikzpicture}
+\kappa^*\hspace{-.2cm}
\begin{tikzpicture}[baseline=.2cm]
	\foreach \x in {1,...,5} 
		\node (\x) at (\x/4,0) [inner sep=0pt] {$\scs\bullet$};
	\foreach \x in {1,...,5} 
		\node at (\x/4,-.2) {$\scscs\x$};
	\draw (1) .. controls (2.25/4,.7) and (3.75/4,.7) ..  node [above=-2pt] {$\scs a$} (5); 
	\draw (2) .. controls (2.75/4,.3) and (3.25/4,.3) ..  node [above=-2pt] {$\scs b$} (4);
\end{tikzpicture}
+\kappa^*\hspace{-.2cm}
\begin{tikzpicture}[baseline=.2cm]
	\foreach \x in {1,...,5} 
		\node (\x) at (\x/4,0) [inner sep=0pt] {$\scs\bullet$};
	\foreach \x in {1,...,5} 
		\node at (\x/4,-.2) {$\scscs\x$};
	\draw (1) .. controls (2.25/4,.7) and (3.75/4,.7) ..  node [above=-2pt] {$\scs b$} (5); 
	\draw (2) .. controls (2.25/4,.25) and (2.75/4,.25) ..  node [above=-2pt] {$\scs a$} (3);
\end{tikzpicture}
+\kappa^*\hspace{-.2cm}
\begin{tikzpicture}[baseline=.2cm]
	\foreach \x in {1,...,5} 
		\node (\x) at (\x/4,0) [inner sep=0pt] {$\scs\bullet$};
	\foreach \x in {1,...,5} 
		\node at (\x/4,-.2) {$\scscs\x$};
	\draw (1) .. controls (2.25/4,.7) and (3.75/4,.7) ..  node [above=-2pt] {$\scs b$} (5); 
	\draw (2) .. controls (2.75/4,.3) and (3.25/4,.3) ..  node [above=-2pt] {$\scs a$} (4);
\end{tikzpicture}\\
&+\kappa^*\hspace{-.2cm}
\begin{tikzpicture}[baseline=.2cm]
	\foreach \x in {1,...,5} 
		\node (\x) at (\x/4,0) [inner sep=0pt] {$\scs\bullet$};
	\foreach \x in {1,...,5} 
		\node at (\x/4,-.2) {$\scscs\x$};
	\draw (1) .. controls (2/4,.35) and (3/4,.35) ..  node [above=-2pt] {$\scs b$} (4); 
	\draw (2) .. controls (3/4,.35) and (4/4,.35) ..  node [above=-2pt] {$\scs a$} (5);
\end{tikzpicture}
+\kappa^*\hspace{-.2cm}
\begin{tikzpicture}[baseline=.2cm]
	\foreach \x in {1,...,5} 
		\node (\x) at (\x/4,0) [inner sep=0pt] {$\scs\bullet$};
	\foreach \x in {1,...,5} 
		\node at (\x/4,-.2) {$\scscs\x$};
	\draw (1) .. controls (2.25/4,.7) and (3.75/4,.7) ..  node [above=-2pt] {$\scs b$} (5); 
	\draw (3) .. controls (3.25/4,.25) and (3.75/4,.25) ..  node [above=-2pt] {$\scs a$} (4);
\end{tikzpicture}
+\kappa^*\hspace{-.2cm}
\begin{tikzpicture}[baseline=.2cm]
	\foreach \x in {1,...,5} 
		\node (\x) at (\x/4,0) [inner sep=0pt] {$\scs\bullet$};
	\foreach \x in {1,...,5} 
		\node at (\x/4,-.2) {$\scscs\x$};
	\draw (1) .. controls (2/4,.35) and (3/4,.35) ..  node [above=-2pt] {$\scs b$} (4); 
	\draw (3) .. controls (3.75/4,.3) and (4.25/4,.3) ..  node [above=-2pt] {$\scs a$} (5);
\end{tikzpicture}
+\kappa^*\hspace{-.2cm}
\begin{tikzpicture}[baseline=.2cm]
	\foreach \x in {1,...,5} 
		\node (\x) at (\x/4,0) [inner sep=0pt] {$\scs\bullet$};
	\foreach \x in {1,...,5} 
		\node at (\x/4,-.2) {$\scscs\x$};
	\draw (1) .. controls (1.75/4,.3) and (2.25/4,.3) ..  node [above=-2pt] {$\scs b$} (3); 
	\draw (4) .. controls (4.25/4,.25) and (4.75/4,.25) ..  node [above=-2pt] {$\scs a$} (5);
\end{tikzpicture}
\end{align*}
and 
\begin{equation*}
\Delta\bigg(
\kappa^*\hspace{-.2cm}
\begin{tikzpicture}[baseline=.2cm]
	\foreach \x in {1,2,3,4} 
		\node (\x) at (\x/4,0) [inner sep=0pt] {$\scs\bullet$};
	\foreach \x in {1,2,3,4} 
		\node at (\x/4,-.2) {$\scscs\x$};
	\draw (1) .. controls (1.25/4,.25) and (1.75/4,.25) ..  node [above=-2pt] {$\scs a$} (2); 
	\draw (2) .. controls (2.75/4,.3) and (3.25/4,.3) ..  node [above=-2pt] {$\scs b$} (4); 
\end{tikzpicture}\bigg)=
\kappa^*\hspace{-.2cm}
\begin{tikzpicture}[baseline=.2cm]
	\foreach \x in {1,2,3,4} 
		\node (\x) at (\x/4,0) [inner sep=0pt] {$\scs\bullet$};
	\foreach \x in {1,2,3,4} 
		\node at (\x/4,-.2) {$\scscs\x$};
	\draw (1) .. controls (1.25/4,.25) and (1.75/4,.25) ..  node [above=-2pt] {$\scs a$} (2); 
	\draw (2) .. controls (2.75/4,.3) and (3.25/4,.3) ..  node [above=-2pt] {$\scs b$} (4); 
\end{tikzpicture}\otimes \kappa_\emptyset^*
+
\kappa^*\hspace{-.2cm}
\begin{tikzpicture}[baseline=.2cm]
	\foreach \x in {1,2,3} 
		\node (\x) at (\x/4,0) [inner sep=0pt] {$\scs\bullet$};
	\foreach \x in {1,2,3} 
		\node at (\x/4,-.2) {$\scscs\x$};
	\draw (1) .. controls (1.25/4,.25) and (1.75/4,.25) ..  node [above=-2pt] {$\scs a$} (2); 
\end{tikzpicture}
\otimes \kappa^*\hspace{-.2cm}
\begin{tikzpicture}[baseline=.1cm]
	\foreach \x in {1} 
		\node (\x) at (\x/4,0) [inner sep=0pt] {$\scs\bullet$};
	\foreach \x in {1} 
		\node at (\x/4,-.2) {$\scscs\x$};
\end{tikzpicture}
+\kappa^*\hspace{-.2cm}
\begin{tikzpicture}[baseline=.2cm]
	\foreach \x in {1,2} 
		\node (\x) at (\x/4,0) [inner sep=0pt] {$\scs\bullet$};
	\foreach \x in {1,2} 
		\node at (\x/4,-.2) {$\scscs\x$};
	\draw (1) .. controls (1.25/4,.25) and (1.75/4,.25) ..  node [above=-2pt] {$\scs a$} (2); 
\end{tikzpicture}
\otimes \kappa^*\hspace{-.2cm}
\begin{tikzpicture}[baseline=.1cm]
	\foreach \x in {1,2} 
		\node (\x) at (\x/4,0) [inner sep=0pt] {$\scs\bullet$};
	\foreach \x in {1,2} 
		\node at (\x/4,-.2) {$\scscs\x$};
\end{tikzpicture}+\kappa^*\hspace{-.2cm}
\begin{tikzpicture}[baseline=.1cm]
	\foreach \x in {1} 
		\node (\x) at (\x/4,0) [inner sep=0pt] {$\scs\bullet$};
	\foreach \x in {1} 
		\node at (\x/4,-.2) {$\scscs\x$};
\end{tikzpicture}\otimes\kappa^*\hspace{-.2cm}\begin{tikzpicture}[baseline=.2cm]
	\foreach \x in {1,2,3} 
		\node (\x) at (\x/4,0) [inner sep=0pt] {$\scs\bullet$};
	\foreach \x in {1,2,3} 
		\node at (\x/4,-.2) {$\scscs\x$};
	\draw (1) .. controls (1.75/4,.3) and (2.25/4,.3) ..  node [above=-2pt] {$\scs b$} (3); 
\end{tikzpicture}+\kappa_\emptyset^* \otimes \kappa^*\hspace{-.2cm}\begin{tikzpicture}[baseline=.2cm]
	\foreach \x in {1,2,3,4} 
		\node (\x) at (\x/4,0) [inner sep=0pt] {$\scs\bullet$};
	\foreach \x in {1,2,3,4} 
		\node at (\x/4,-.2) {$\scscs\x$};
	\draw (1) .. controls (1.25/4,.25) and (1.75/4,.25) ..  node [above=-2pt] {$\scs a$} (2); 
	\draw (2) .. controls (2.75/4,.3) and (3.25/4,.3) ..  node [above=-2pt] {$\scs b$} (4); 
\end{tikzpicture}.
\end{equation*}
\end{example}

\subsection{A realization of $\SC^*$}

A priori, it is not clear that $\NCSym^*$ or $\SC^*$ should have a realization as a space of functions in commuting variables.
Here, we summarize  some results of \cite{HNT} giving such a realization,  and remark that the variables must satisfy relations closely related to the definition of $\cS_n(q)$.

Let $x_{ij}$, for $i,j\ge 1$, be commuting variables satisfying
the relations
\begin{equation}\label{relxij}
x_{ij}x_{ik}=0 \quad \text{and}  \quad
x_{ik}x_{jk}=0 \ \text{for all $i,j,k$.}
\end{equation}
For a permutation $\sigma\in\SG_n$, define
\begin{equation}\label{Msigma}
\Mper_\sigma = \sum_{i_1 < \cdots < i_n}
            x_{i_1\, i_{\sigma(1)}}\cdots x_{i_n\, i_{\sigma(n)}}.
\end{equation}
It is shown in \cite{HNT} that these polynomials span a (commutative, cofree)
Hopf algebra, denoted by $\SG QSym$.

For $\alpha\in \SG_m$ and $\beta\in \SG_n$ we define coefficients $C_{\alpha,\beta}^{\gamma}$ 
\begin{equation}
\label{defC}
\Mper_\alpha \Mper_\beta := \sum_{\gamma} C_{\alpha,\beta}^{\gamma}
\Mper_\gamma.
\end{equation}
which can be computed by the following process:
\begin{description}
\item[Step 1.] Write $\alpha$ and $\beta$ as products of disjoint cycles.
\item[Step 2.] For each subset $A\subseteq [m+n]$ with $m$ elements, renumber  $\alpha$ using the unique order-preserving bijection $\st_A^{-1}:[m]\rightarrow A$, and renumber $\beta$ with the unique order preserving bijection $\st_{A^c}^{-1}:[n] \rightarrow A^c$.
\item[Step 3.] The resulting permutation $\gamma$ gives a term $\Mper_\gamma$ in the product $\Mper_\alpha \Mper_\beta$.
\end{description}
Thus, $C_{\alpha,\beta}^\gamma$ is the number of ways to obtain $\gamma$ from $\alpha$ and $\beta$ using this process.
\begin{example}
If $\alpha=(1)(2)=12$, $\beta=(31)(2)=321$, then Step 2 yields
\begin{equation}
\label{12-321b}
\begin{split}
\overset{A=\{1,2\}}{(1)(2)(53)(4)},\ \overset{A=\{1,3\}}{(1)(3)(52)(4)},\ \overset{A=\{1,4\}}{(1)(4)(52)(3)},\ \overset{A=\{1,5\}}{(1)(5)(42)(3)},\ \overset{A=\{2,3\}}{(2)(3)(51)(4)},\\
\overset{A=\{2,4\}}{(2)(4)(51)(3)},\ \overset{A=\{2,5\}}{(2)(5)(41)(3)},\ \overset{A=\{3,4\}}{(3)(4)(51)(2)},\ \overset{A=\{3,5\}}{(3)(5)(41)(2)},\ \overset{A=\{4,5\}}{(4)(5)(31)(2)},
\end{split}
\end{equation}
and thus $C_{(1)(2),(31)(2)}^{(51)(2)(3)(4)}=3$.
\end{example}

Another interpretation of this product is given by the dual point of
view: $C_{\alpha,\beta}^{\gamma}$ is the number of ways of getting
$(\alpha,\beta)$ as the standardized words of pairs $(a,b)$ of two
complementary subsets of cycles of $\gamma$.
For example, with $\alpha=12$, $\beta=321$, and $\gamma=52341$, one has three
solutions for the pair $(a,b)$, namely
\begin{equation}
((2)(3), (4)(51)),\ \  ((2)(4), (3)(51)),\ \ ((3)(4), (2)(51)).
\end{equation}

\begin{remark}
Each function in $\SG QSym$ can be interpreted as a function on matrices by evaluating $x_{ij}$ at the $(i,j)$-th entry of the matrix (or zero if the matrix does not have an $(i,j)$ entry).  {}From this point of view, $\SG QSym$ intersects the ring of class functions of the wreath product $C_r\wr S_n$ in such a way that it contains the ring of symmetric functions as a natural subalgebra.
\end{remark}

For a permutation $\sigma\in\SG_n$, let $\csupp(\sigma)$ be the partition
$\mu$ of the set $[n]$ whose blocks are the supports of the cycles of
$\sigma$. The sums
\begin{equation}
\upi_{\mu} := \sum_{\csupp(\sigma)=\mu} \Mper_\sigma
\end{equation}
span a Hopf subalgebra $\Pi QSym$ of $\SG QSym$, which is isomorphic to the
graded dual of $\WSym$. 
Indeed, from the product rule of the $\Mper_\sigma$ given in
(\ref{defC}), it follows that
\begin{equation}
\upi_{\mu}\upi_{\nu} = \sum C_{\mu,\nu}^\lambda \upi_\lambda,
\end{equation}
where $C_{\mu,\nu}^\lambda$ is the number of ways of splitting the parts of
$\lambda$ into two subpartitions whose standardized words are $\mu$ and $\nu$.
For example,
\begin{equation}
\upi_{ 124|3}\upi_{1} =
\upi_{ 124|3|5 } +
2\upi_{ 125|3|4} +
\upi_{ 135|4|2} +
\upi_{ 235|4|1}.
\end{equation}

Hence, the basis $\upi_\mu$ has the same product rule as $\bm_\mu^*$.
However, it does not have the same coproduct. To find the correct identification we need the $\bp_\lambda$ basis
introduced in \cite{SR}. Let
  $$\bp_\lambda =\sum_{\mu\ge \lambda} \bm_\mu,$$
where $\mu\ge \lambda$ means that $\lambda$ refines $\mu$. As shown in \cite{BHRZ,SR}
  $$\bp_\lambda \bp_\mu = \bp_{\lambda|\mu}$$
  and following the notation of (\ref{mcomult})
  $$\Delta(\bp_\lambda)=\sum_{J\subseteq [\ell(\lambda)]} \bp_{\mathrm{st}(\lambda_J)}\otimes \bp_{\mathrm{st}(\lambda_{J^c})}.$$
These are precisely the operations we need to give the isomorphism 
$$\begin{array}{rccc}\theta\colon& \Pi QSym&\longrightarrow & \NCSym^*\\  & \upi_\mu & \mapsto & \bp^*_\mu.\end{array}$$  
The isomorphism in Theorem~\ref{q=2Correspondence} maps $\kappa_\mu\mapsto\bm_\mu$, and the dual map is $\bm^*_\mu\mapsto\kappa^*_\mu$.
Hence, if we define
  $$V_\mu=\theta^{-1}(\bm_\mu^*) = \sum_{\nu\le\mu} \upi_\nu,$$
then we obtain the following theorem.

\begin{theorem}
For $q=2$, the function
$$\begin{array}{rccc} \ch: & \SC^* & \longrightarrow & \Pi QSym\\  & \kappa_\mu^* & \mapsto & V_\mu\end{array}$$
is a Hopf algebra isomorphism.
\end{theorem}

\begin{remark}
For general $q$ one needs a colored version of $ \Pi QSym$. This can be done in the same spirit of Section~\ref{ColoredCorrespondence} and we leave it to the reader.
\end{remark}

\section{Appendix}

Supercharacter theory has continued its development since we began our work. In addition to the references above, we call attention to a selection of recent papers.  In particular, the results of this paper have been generalized to unipotent subgroups of type $D$ \cite{Be11}, there has been some exploration into further decomposing these supercharacters \cite{Le10,Ma11} and further analysis of their underlying combinatorics \cite{MaA11,MaP11}.

In addition to the above results, the American Institute of Mathematics workshop generated several items that might be of interest to those who would like to pursue these thoughts further. 

\subsection{Sage}

A Sage package has been written, and is described at

\medskip

\noindent \textsf{http://garsia.math.yorku.ca/\textasciitilde saliola/supercharacters/}

\medskip

\noindent It has a variety of functions, including the following.
\begin{itemize}
\item It can use various bases, including the supercharacter basis and the superclass functions basis, 
\item It can change bases,
\item It computes products, coproducts and antipodes in this Hopf algebra,
\item It computes the inner tensor products (pointwise product) and restriction in the ring of supercharacters,
\item It gives the supercharacter tables for $\UT_n(q)$.
\end{itemize}

\subsection{Open problems}

There is a list of open problems related to this subject available at

\medskip

\noindent \textsf{http://www.aimath.org/pastworkshops/supercharacters.html}

\medskip


 M. Aguiar, University of Texas A\&M,  \textsf{maguiar@math.tamu.edu}
 
  C. Andr\'e, University of Lisbon, \textsf{caandre@fc.ul.pt}
  
 C. Benedetti, York University, \textsf{carobene@mathstat.yorku.ca}
 
N. Bergeron, York University, supported by CRC and NSERC, \textsf{bergeron@yorku.ca}

 Z. Chen, York University, \textsf{czhi@mathstat.yorku.ca}
 
 P. Diaconis, Stanford University, supported by NSF DMS-0804324, \textsf{diaconis@math.stanford.edu}
 
A. Hendrickson, Concordia College, \textsf{ahendric@cord.edu}

 S. Hsiao, Bard College, \textsf{hsiao@bard.edu}
 
I.M. Isaacs, University of Wisconsin-Madison, \textsf{isaacs@math.wisc.edu}

A. Jedwab, University of Southern California, supported by NSF DMS 07-01291, \textsf{jedwab@usc.edu}

 K. Johnson, Penn State Abington, \textsf{kwj1@psuvm.psu.edu }
 
 G. Karaali, Pomona College, \textsf{gizem.karaali@pomona.edu}
 
A. Lauve, Loyola University, \textsf{lauve@math.luc.edu}

T. Le, University of Aberdeen, \textsf{t.le@abdn.ac.uk} 

S. Lewis, University of Washington, supported by NSF DMS-0854893, \textsf{stedalew@u.washington.edu}

H. Li,  Drexel University, supported by NSF DMS-0652641, \textsf{huilan.li@gmail.com}

K. Magaard, University of Birmingham, \textsf{k.magaard@bham.ac.uk}

 E. Marberg, MIT, supported by NDSEG Fellowship, \textsf{emarberg@math.mit.edu}
 
 J.-C. Novelli, Universit\'e Paris-Est Marne-la-Vall\'ee, \textsf{novelli@univ-mlv.fr}
 
 A. Pang, Stanford University, \textsf{amypang@stanford.edu}
 
 F. Saliola, Universit\'e du Qu\'ebec \'a Montr\'eal, supported by CRC, \textsf{saliola@gmail.com}
 
 L. Tevlin, New York University, \textsf{ltevlin@nyu.edu}
 
 J.-Y. Thibon, Universit\'e Paris-Est Marne-la-Vall\'ee, \textsf{jyt@univ-mlv.fr}
 
 N. Thiem, University of Colorado at Boulder, supported by NSF DMS-0854893, \textsf{thiemn@colorado.edu}
 
 V. Venkateswaran, Caltech University, \textsf{vidyav@caltech.edu}
 
 C.R. Vinroot, College of William and Mary, supported by NSF DMS-0854849, \textsf{vinroot@math.wm.edu}
 
 N. Yan, \textsf{ning.now@gmail.com}
 
 M. Zabrocki, York University, \textsf{zabrocki@mathstat.yorku.ca}

\normalsize

\footnotesize
\begin{thebibliography}{abc}
%
\bibitem{ABS06} M. Aguiar, N. Bergeron and  F. Sottile, {\it Combinatorial Hopf Algebras and generalized Dehn--Sommerville relations}, Compositio Math. \textbf{142} (2006), 1--30.
%
\bibitem{ABT} M. Aguiar, N. Bergeron and N. Thiem,  {\it A Hopf monoid from the representation theory of the finite group of
unitriangular matrices},    in preparation.
%
\bibitem{AM} M. Aguiar and S. Mahajan, {\it Coxeter groups and Hopf algebras}, Amer. Math. Soc. Fields Inst. Monogr. \textbf{23} (2006).
%
\bibitem{AM10} M. Aguiar and S. Mahajan, {\it Monoidal functors, species and Hopf algebras}, CRM Monograph Series, 29, Amer. Math. Soc.  Providence, 2010.
%
\bibitem{An95} C. Andr\'e, {\it Basic characters of the unitriangular group}, J.  Algebra \textbf{175} (1995), 287--319.
%
\bibitem{AN06} C. Andr\'e and A Neto, {\it Super-characters of finite unipotent groups of types $B_n$, $C_n$ and $D_n$}, J. Algebra \textbf{305} (2006), 394--429.
%
\bibitem{ADS} E. Arias-Castro, P. Diaconis and R. Stanley, {\it  A super-class walk on upper-triangular matrices}, J. Algebra \textbf{278} (2004), 739--765.
%
\bibitem{Be11} C. Benedetti, {\it Combinatorial Hopf algebra of supercharacters of type $D$}, arXiv:1111.6185.
%
\bibitem{BH} N. Bergeron and C. Hohlweg, {\it Coloured peak algebras and Hopf algebras},  J. Algebraic Combin.  {\bf 24}  (2006),  no. 3, 299--330.
%
\bibitem{BHRZ}  N. Bergeron, C. Hohlweg, M. Rosas and  M. Zabrocki, {\it Grothendieck bialgebras, partition lattices, 
and symmetric functions in noncommutative variables}, Electron. J. Combin.  {\bf 13}  (2006),  no. 1, Research Paper 75, 19 pp. 
%
\bibitem{BLL} N. Bergeron, T. Lam and H. Li, {\it Combinatorial Hopf algebras and Towers of Algebras -- Dimension, Quantization and Functorality}, to appear in  Algebr. Represent. Theory, DOI: 10.1007/s10468-010-9258-y.
%
\bibitem{BRRZ} N. Bergeron, C. Reutenauer, M. Rosas and M. Zabrocki,
{\it Invariants and coinvariants of the symmetric groups in noncommuting variables}, Canad. J. Math.  {\bf 60}  (2008),  no. 2, 266--296.
%
\bibitem{BZ} N. Bergeron and   M. Zabrocki, {\it The Hopf algebras of symmetric functions and quasi-symmetric functions 
in non-commutative variables are free and co-free}, J. Algebra Appl.  \textbf{8}  (2009),  no. 4, 581--600.
%
\bibitem{CP94} V. Chari and A Pressley,  {\it A guide to quantum groups}, Cambridge University Press, Cambridge, 1994.
%
\bibitem{CK} A. Connes and D. Kreimer. {\it Hopf algebras, renormalization and noncommutative geometry},
Commun. Math. Phys. {\bf 199} (1998) 203--242.
%
\bibitem{DI08} P. Diaconis and M Isaacs, {\it Supercharacters and superclasses for algebra groups},  Trans. Amer. Math. Soc.  \textbf{360}  (2008), 2359--2392.
%
\bibitem{Do} P. Doubilet, {\it On the foundations of combinatorial theory. VII:
Symmetric functions through the theory of distribution and occupancy},
Studies in Applied Math. \textbf{51} (1972), 377--396.
%
\bibitem{Dr}  V.G. Drinfeld, {\it Quantum groups} in ``Proceedings ICM'', Berkeley, Amer. Math. Soc. (1987) 798--820.
%
\bibitem{GS} D.D.  Gebhard and B.E. Sagan, {\it A chromatic symmetric function in noncommuting variables}, J. Algebraic Combin. {\bf 13}  (2001),  no. 3, 227--255.
%
\bibitem{Ge77} L. Geissinger, {\it Hopf algebras of symmetric functions and class functions}, Lecture Notes in Math. \textbf{579} (1977) 168--181.
%
\bibitem{NCSF1} I.M. Gelfand, D. Krob, A. Lascoux, B. Leclerc, V.~S. Retakh, and   J.-Y. Thibon,
{\it Noncommutative symmetric functions}, Adv. Math. {\bf 112} (1995), 218--348.
%
\bibitem{He08} A.O.F. Hendrickson, {\it Supercharacter theory constructions corresponding to Schur ring products}, to appear in Comm. Algebra.
%
\bibitem{HNT}  F. Hivert, J.-C. Novelli  and J.-Y. Thibon, {\it Commutative combinatorial Hopf algebras}, J. Algebraic Combin.  {\bf 28}  (2008),  no. 1, 65--95.
%
\bibitem{JR79} S. A. Joni and G. C. Rota, {\it Coalgebras and bialgebras in combinatorics},  Stud. Appl. Math. \textbf{61} (1979), 93--139.
%
%
\bibitem{LM} A. Lauve and M. Mastnak, {\it The primitives and antipode in the Hopf algebra of symmetric functions in noncommuting variables}, Adv. in Appl. Math. \textbf{47} (2011),  536--544.
%
\bibitem{Le10} T. Le, {\it Supercharacters and pattern subgroups in the upper triangular groups}, arXiv:1008.2269.
%
\bibitem{Li} H. Li, {\it Algebraic Structures of Grothendieck Groups of a Tower of Algebras,} Ph.D. Thesis, York University, 2007.
%
\bibitem{Mcd} {I.G. Macdonald}, {\it Symmetric functions and Hall polynomials}, 2nd ed., Oxford University Press, 1995.
%
\bibitem{Ma11} E. Marberg, {\it Iterative character constructions for algebra groups}, Adv. Math. \textbf{228} (2011), 2743--2765.
%
\bibitem{MaA11} E. Marberg, {\it Combinatorial methods of character enumeration for the unitriangular group}, J. Algebra \textbf{345} (2011), 295--323.
%
\bibitem{MaP11} E. Marberg, {\it Heisenberg characters, unitriangular groups, and Fibonacci numbers}, arXiv:1105.1003. 
%
\bibitem{MT09} E. Marberg and N. Thiem, {\it Superinduction for pattern groups}, {J. Algebra} \textbf{321} (2009), 3681--3703.
%
\bibitem{Mo93} S. Montgomery, {\it Hopf algebras and their actions on rings,}  CBMS Regional Conference Series in Mathematics \textbf{82}  (1993) Washington DC.
%
\bibitem{NT06} {J.-C. Novelli} and {J.-Y. Thibon}, {\it Polynomial realizations of some trialgebras}, FPSAC'06. Also preprint ArXiv:math.CO/0605061.
%
%
\bibitem{SR} {M.H. Rosas} and {B.E. Sagan}, {\it Symmetric functions in noncommuting variables}, Trans. Amer. Math. Soc.  {\bf 358}  (2006),  no. 1, 215--232
%
\bibitem{SS93} S. Shnider and S. Sternberg, {\it Quantum groups: From coalgebras to Drinfeld algebras, a guided tour}, Graduate Texts in Mathematical Physics, II, International Press (1993) Cambridge, MA. 
%
\bibitem{Th} {N. Thiem}, {\it Branching rules in the ring of superclass functions  of unipotent upper-triangular matrices},
J. Algebraic Combin.  {\bf 31}  (2010),  no. 2, 267--298.
%
\bibitem{TV09} N. Thiem and V. Venkateswaran, {\it Restricting supercharacters of the finite group of unipotent uppertriangular matrices}, {Electron. J. Combin.} \textbf{16}(1) Research Paper 23 (2009), 32 pages.
%
\bibitem{Wo} {M.C.  Wolf}, {\it Symmetric functions of non-commuting elements},  Duke Math. J. {\bf 2} (1936) 626--637.
 %
 \bibitem{Ya01} N. Yan, {\it Representation theory of the finite unipotent linear groups}, Unpublished Ph.D. Thesis, Department of Mathematics, University of Pennsylvania, 2001.
  %
 \bibitem{Ze} A. Zelevinsky. {\it Representations of Finite Classical Groups},  Springer Verlag, 1981.
\end{thebibliography}
\end{document}